\documentclass[hidelinks,onefignum,onetabnum]{siamart220329}

\usepackage{lipsum}
\usepackage{amsfonts}
\usepackage{graphicx}
\usepackage{epstopdf}
\usepackage{algorithmic}
\ifpdf
  \DeclareGraphicsExtensions{.eps,.pdf,.png,.jpg}
\else
  \DeclareGraphicsExtensions{.eps}
\fi

\newsiamremark{remark}{Remark}
\newsiamremark{hypothesis}{Hypothesis}
\crefname{hypothesis}{Hypothesis}{Hypotheses}
\newsiamthm{claim}{Claim}

\headers{ Moving higher-order Taylor approximations method}{Yassine Nabou and Ion Necoara}

\title{Moving higher-order Taylor  approximations method for smooth constrained minimization problems\thanks{Submitted to the editors \date.
\funding{ITN-ETN project TraDE-OPT funded by the European Union’s Horizon 2020 Research and Innovation Programme under the Marie Skolodowska-Curie grant agreement No. 861137;  
UEFISCDI, Romania, PN-III-P4-PCE-2021-0720, under project L2O-MOC, nr. 70/2022.}}}

\author{Yassine Nabou\thanks{Automatic Control and Systems Engineering Department, University Politehnica Bucharest, Romania, 
  (\email{yassine.nabou@stud.acs.upb.ro})}
\and Ion Necoara\thanks{ Automatic Control and Systems Engineering Department, University Politehnica Bucharest and Gheorghe Mihoc - Caius Iacob Institute of Mathematical Statistics and Applied Mathematics of the Romanian Academy, Bucharest, Romania,
  (\email{ion.necoara@upb.ro})}}

\usepackage{amsopn}

\ifpdf
\hypersetup{
  pdftitle={Moving higher-order Taylor approximations method for smooth problems},
  pdfauthor={Yassine Nabou, and Ion Necoara}
}
\fi

\newcommand{\norm}[1]{\lVert#1\rVert}
\newtheorem{assumption}{Assumption}
\usepackage{color}
\usepackage{amssymb}
\usepackage{amsmath}

\newcommand{\red}{\textcolor{black}}
\usepackage{comment}
\usepackage{multirow}
\usepackage{graphicx}
\usepackage{caption}
\usepackage{subcaption}
\begin{document}

\maketitle

\begin{abstract}
In this paper, we develop a higher-order method for solving composite (non)convex minimization problems with smooth (non)convex functional constraints. At each iteration, our method approximates the smooth part of the objective function and of the constraints by higher-order Taylor approximations, leading to a moving Taylor approximation method (MTA). We present convergence guarantees for the MTA algorithm for both, nonconvex and convex problems. In particular, when the objective and the constraints are nonconvex functions, we prove that the sequence generated by MTA algorithm converges globally to a KKT point. Moreover, we derive  convergence rates in the iterates when the problem’s data satisfy the Kurdyka-Lojasiewicz (KL) property. Further, when the objective function is (uniformly) convex and the constraints are also convex, we provide (linear/superlinear) sublinear convergence rates for our algorithm. Finally, we present an efficient implementation of the proposed algorithm and compare it with existing methods from the literature.
\end{abstract}



\section{Introduction}\label{sec:1}
In this paper, we consider the following composite optimization problem with functional constraints:
\begin{align}
\label{eq:problem}
 \min_{x \in \mathbb{E}} F(x)&: = F_0(x) + h(x)\\
 \text{s.t.}&: \; F_i(x) \leq 0 \quad \forall i=1:m, \nonumber 
\end{align}
\noindent where $F_i : \mathbb{E} \to \mathbb{R}$, for $i=0:m$, are continuous differentiable functions and $h: \mathbb{E} \to \bar{\mathbb{R}} $ is proper and convex function (e.g., the indicator function of some convex set). Here, $\mathbb{E}$ is a finite-dimensional real vector space. Nonlinear programming problems have a long and rich history (see, for example, the monograph  \cite{Ber:99}), because they model many practical applications. Some of the most common are power systems, engineering design, control, signal and image processing, machine learning and statistics, see e.g. \cite{Ban:74,ChTaPa:07,PalChi:07,RigTon:11,SaDa:95,SheZha:04}. Several algorithms with complexity guarantees have been developed for solving \eqref{eq:problem}. For example, the successive linearization methods \cite{BoTo:95,MesBau:21,Zou:60} are based on solving a sequence of quadratic regularized subproblems over a set of linearized approximations of
the constraints (known as SCP/SQP type algorithms). Furthermore, the augmented Lagrangian (AL) methods reformulate the nonlinear problem \eqref{eq:problem} via unconstrained minimization problems (known as methods of multipliers), see e.g., \cite{Ber:82,Hes:69}. For the particular class of constrained optimization problems with smooth data,  \cite{AuShTe:10} introduces a moving ball approximation method (MBA) that approximates the objective function with a quadratic and the feasible set by a sequence of appropriately defined balls. The authors provide asymptotic convergence guarantees for the sequence generated by MBA when the data is (non)convex, and linear convergence if the objective function is strongly convex. Later, several papers considered variants of the MBA algorithm; see \cite{BolChe:20,BoPa:16,YuPoLu:21}. For example, in \cite{YuPoLu:21} the authors present a line search MBA algorithm for difference-of-convex minimization problems and derive asymptotic convergence in the nonconvex settings and local convergence in the iterates when a special potential function related to the objective satisfies the Kurdyka-Lojasiewicz (KL) property. In \cite{BolChe:20}, the authors consider convex composite minimization problems that cover, in particular, problems of the form \eqref{eq:problem}, and use a similar MBA type scheme for solving such problems, deriving a sublinear convergence rate for it. Note that all these previous methods are \textit{first-order methods}, and despite their empirical success in solving difficult optimization problems, their convergence speed is known to be slow.

\medskip 

\noindent A natural way to ensure faster convergence rates is to increase the power of the oracle, i.e., to use higher-order information (derivatives) to build a higher-order (Taylor) model {(see  \cite{CaGoTo:22} for a detailed exposition)}. For example, \cite{NesPol:06} derives the first global convergence rate of cubic regularization of Newton method for unconstrained smooth minimization problems with the hessian Lipschitz continuous (i.e., using a second-order oracle). Paper  \cite{NesPol:06} derives global convergence guarantees (in function value) in the convex case of order $\mathcal{O}(k^{-2})$. Higher-order methods are recently popular due to their performance in dealing with ill conditioning and fast rates of convergence. However, the main obstacle in the implementation of these (higher-order) methods lies in the complexity of the corresponding model approximation formed by a high-order multidimensional polynomial, which may be difficult to handle and minimize (see, for example, \cite{BiGaMa:17,CaGoTo:19}). Nevertheless, for convex unconstrained smooth problems, \cite{Nes:20} proved that a regularized Taylor approximation becomes convex when the corresponding regularization parameter is sufficiently large. This observation opened the door for using higher-order Taylor approximations to different structured problems (see, for example, \cite{DoiNes:20,DoiNes:21,GasDvu:18,GraNes:20,NabNec:23}). {Several papers have already proposed higher-order methods for solving composite optimization problems of the form \eqref{eq:problem} with complexity guarantees, see e.g., \cite{BiGaMa:16,Mar:17,DoiNes:21,CaGoTo:22}}. For example, in a recent paper \cite{DoiNes:21}, the authors consider fully composite problems (which cover, as a particular case, \eqref{eq:problem}), and when the data are $p$ times continuously differentiable with the $p$th derivative Lipschitz and uniformly convex, they derive  linear convergence  in function values. In \cite{BiGaMa:16}, the authors derive worst-case complexity bounds for computing approximate first-order critical points using higher-order derivatives for problem \eqref{eq:problem} with nonlinear equality constraints. {Paper \cite{Mar:17} also considers problem \eqref{eq:problem} with nonlinear equality constraints and  employs an approach wherein the objective is approximated with a model of arbitrarily high-order. Each iteration requires  the computation of an approximate KKT point for the subproblem, devoiding of any constraint qualification condition. The authors show that their  scheme converges to an approximate KKT point within $\mathcal{O}\left(\epsilon^{-\frac{p+\beta}{p+1-\beta}}\right)$ iterations, where $\beta \in [0,1]$.}

\vspace{0.2cm}

\noindent  In this paper, we develop a higher-order method for solving composite (non)convex minimization problems with smooth (non)convex functional constraints \eqref{eq:problem}. At each iteration, our method approximates the smooth part of the objective function and of the constraints by a higher-order Taylor approximation, leading to a moving Taylor approximation method. We present convergence guarantees for our algorithm for both nonconvex and convex problems.

\medskip

\begin{table}
\centering
\begin{tabular}{|l|l|l|}
\hline
                                                           & Convergence rates                      & Theorem \\ \hline
\multirow{2}{*}{Nonconvex}                                 & \begin{tabular}[c]{@{}l@{}}Measure of optimality (see section \ref{sec:kkt}):\\ $\min\limits_{i=1:k}{\cal M}(x_{i})\leq \mathcal{O}\left(k^{-\min\left(\frac{p}{p+1}, \frac{q}{q+1}\right)}\right)$\end{tabular} &  \ref{th:2}    \\ \cline{2-3} 
                                                           & $x_k\to x^*$  sublinearly or linearly under KL                       &  \ref{th:kl}    \\ \hline
Convex                                                     & $F(x_k) -  F^* \leq  \mathcal{O}\left(k^{-\min(p,q)}\right)$                                                                                                               & \ref{th:cvx}     \\ \hline
\begin{tabular}[c]{@{}l@{}}Uniformly convex\end{tabular} & $F(x_k)\to F^*$   linearly or superlinearly                                                                                                                                   & \ref{th:unif}    \\ \hline
\end{tabular}
\caption{Summary of the convergence results obtained in this paper for  MTA.}\label{tab:0}
\end{table}


\noindent \textit{Contributions.}  Our main contributions are as follows:\\
\noindent (i) We introduce a new moving Taylor approximation (MTA) method for solving problem \eqref{eq:problem} \red{that requires starting from an initial feasible point}. Our framework is flexible in the sense that we can approximate the objective function and the constraints with higher-order Taylor approximations of different degrees (i.e., we can approximate the smooth part of the objective function with a Taylor approximation of degree $p$ and the constraints with a Taylor approximation of degree~$q$). 

\vspace{0.2cm}

\noindent (ii) We derive convergence guarantees for MTA algorithm for (non)convex problems with smooth (non)convex functional constraints. More precisely, when the data is nonconvex, we show that the iterates generated by MTA converge to a KKT point and the convergence rate is of order $\mathcal{O}\left(k^{-\min\left(\frac{p}{p+1}, \frac{q}{q+1}\right)}\right)$, where $k$ is the iteration counter and $p$ and $q$ are the degrees of the Taylor approximations for objective and constraints, respectively. When the data of the problem are semialgebraic, we derive (linear) sublinear convergence rates in the iterates (depending on the parameter of the KL property). Moreover, for convex problems, we  derive a global sublinear convergence rate of order $\mathcal{O}\left(k^{-\min(p,q)}\right)$ in function values, and additionally, if the objective function is uniformly convex, we derive a (superlinear) linear convergence rate (depending on the degree of the uniform convexity). The convergence rates obtained in this paper are summarized in Table \ref{tab:0}. 

\vspace{0.2cm}

\noindent (iii) Note that the subproblem we need to solve at each iteration of MTA  is usually nonconvex and it can have local minima. However, we show  for $p, q \leq 2$ that our approach is implementable, since this subproblem is equivalent to minimizing an explicitly written convex function over a convex set that can be solve using efficient convex optimization tools. We believe that this is an additional step towards practical implementation of higher-order (tensor) methods in smooth nonconvex optimization problems with smooth nonconvex functional constraints. 

\medskip 

\noindent  Besides providing a unifying   analysis of  higher-order  methods, in special cases, where complexity bounds are known for some particular   algorithms, our convergence results recover the  existing bounds. For example, for $p=q=1$, we recover the convergence results obtained in \cite{AuShTe:10,YuPoLu:21} in the nonconvex setting. We also recover the sublinear convergence rate in the convex case derived in \cite{BolChe:20} for $p=q=1$, as well as the linear convergence rate in function values obtained in \cite{DoiNes:21}, but we only assume uniform convexity on the objective and not on the constraints.

\section{Notations and preliminaries}\label{sec:2}

We denote a finite-dimensional real vector space with $\mathbb{E}$ and $\mathbb{E}^{*}$ its dual space composed of linear functions on $\mathbb{E}$. For  $s\in\mathbb{E}^{*}$, the value of $s$ at point $x\in\mathbb{E}$ is denoted by $\langle s,x\rangle$.  Using a self-adjoint positive-definite operator $B:\mathbb{E}\rightarrow \mathbb{E}^{*}$, we endow these spaces with conjugate Euclidean norms:
\begin{align*}
\norm{x}=\langle Bx,x\rangle^{\frac{1}{2}},\quad x\in \mathbb{E},\qquad \norm{g}_{*}=\langle g,B^{-1}g\rangle^{\frac{1}{2}},\quad g\in \mathbb{E}^{*}.
\end{align*}   
For a twice differentiable function $\phi$ on a convex and open domain $\text{dom}\;\phi \subseteq \mathbb{E}$, we denote by $\nabla \phi(x)$ and $\nabla^{2} \phi(x)$ its gradient and hessian evaluated at  $x\in \text{dom}\;\phi$, respectively. Then, $\nabla \phi(x)\in\mathbb{E}^{*}$ and $\nabla^{2}\phi(x)h\in\mathbb{E}^{*}$ for all $x\in\text{dom}\;\phi$, $h\in\mathbb{E}$.  Throughout the paper, we consider $p,q$ positive integers.  In what follows, we often work with directional derivatives of function $\phi$ at $x$ along directions $h_{i}\in \mathbb{E}$ of order $p$, $D^{p} \phi (x)[h_{1},\cdots,h_{p}]$, with $i=1:p$.  If all the directions $h_{1},\cdots,h_{p}$ are the same, we use the notation
$ D^{p} \phi(x)[h]$, for $h\in\mathbb{E}.$
Note that if $\phi$ is $p$ times differentiable, then $D^{p} \phi (x)$ is a symmetric $p$-linear form. Then, its norm is defined as \cite{Nes:20}:

\vspace{-0.2cm}

\begin{align*}
\norm{D^{p} \phi (x)}= \max\limits_{h\in\mathbb{E}} \left\lbrace D^{p} \phi (x)[h]^{p} :\norm{h}\leq 1 \right\rbrace.  
\end{align*}

\vspace{-0.1cm}

\noindent Further, the Taylor approximation of order $p$ of $\phi$ at $x\in \text{dom}\;\phi$ is denoted with:

\vspace{-0.4cm}

\begin{align*}
    T_p^{\phi}(y; x)= \phi(x) + \sum_{i=1}^{p} \frac{1}{i !} D^{i} \phi(x)[y-x]^{i} \quad \forall y \in \mathbb{E}.
\end{align*}

\vspace{-0.1cm}

\noindent Let $\phi: \mathbb{E}\mapsto \mathbb{R}$ be a $p$ differentiable function on $\mathbb{E}$. Then, the $p$ derivative is Lipschitz continuous if there exist a constant $L_p^{\phi} > 0$ such that:
   	\begin{equation} \label{eq:001}
    	\| D^p \phi(x) - D^p \phi(y) \| \leq L_p^{\phi} \| x-y \| \quad  \forall x,y \in \mathbb{E}.
    	\end{equation}  

\vspace{-0.1cm}
     
\noindent It is known that if \eqref{eq:001} holds,  then the residual between the function  and its Taylor approximation can be bounded \cite{Nes:20}:

\vspace{-0.5cm}

    \begin{equation}\label{eq:TayAppBound}
    \vert \phi(y) - T_p^{\phi}(y;x) \vert \leq  \frac{L_p^{\phi}}{(p+1)!} \norm{y-x}^{p+1} \quad  \forall x,y \in \mathbb{E}.
    \end{equation}

\vspace{-0.2cm}
  
\noindent  If $p \geq 2$, we also have the following inequalities valid for all $ x,y \in \mathbb{E}$:

\vspace{-0.5cm}

    \begin{align} \label{eq:TayAppG1}
    &\| \nabla \phi(y) - \nabla T_p^{\phi}(y;x) \|_* \leq \frac{L_p^{\phi}}{p!} \|y-x \|^p, \\
    \label{eq:TayAppG2}
    &\|\nabla^2 \phi(y) - \nabla^2 T_p^{\phi}(y;x) \| \leq \frac{L_p^{\phi}}{(p-1)!} \| y-x\|^{p-1}.
    \end{align}

\vspace{-0.2cm}
    
\noindent For the Hessian, the norm defined in \eqref{eq:TayAppG2} corresponds to the spectral norm of self-adjoint linear operator (maximal module of all eigenvalues computed w.r.t. $B$). We say that $\phi$ is uniformly convex of degree $\theta$ with constant $\sigma$ if:

\vspace{-0.4cm}

\begin{align*}
    \phi(y) \geq \phi(x) +\langle \nabla \phi(x), y-x \rangle + \frac{\sigma}{\theta}\|y-x\|^\theta \quad \forall x,y\in\mathbb{E}.
\end{align*}

\vspace{-0.2cm}

\noindent For a convex function $h:\mathbb{E} \mapsto \mathbb{R}$, we denote by $\partial h(x)$ its subdifferential at $x$ that is defined as $\partial h(x) : =\{\lambda\in \mathbb{E}^*  : h(y) \geq h(x) + \langle \lambda , y-x\rangle\;\forall y\in \mathbb{E}\}$.  Denote $S_f(x) := \text{dist}(0,\partial f(x))$. Let us also recall the definition of a function satisfying the \textit{Kurdyka-Lojasiewicz (KL)} property (see \cite{BolDan:07} for more details). 

\begin{definition}
\label{def:kl}
\noindent A proper lower semicontinuous  function $f: \mathbb{E}\rightarrow (-\infty , +\infty]$ satisfies  \textit{Kurdyka-Lojasiewicz (KL)} property on the compact set $\Omega \subseteq \text{dom} \; f$ on which $f$ takes a constant value $f_*$ if there exist $\delta, \epsilon >0$ such that   one has:
\begin{equation}\label{eq:kl}
\kappa' (f(x) - f_*) \cdot  S_{f}(x)  \geq 1  \quad   \forall x\!:  \emph{dist}(x, \Omega) \leq \delta, \;  f_* < f(x) < f_* + \epsilon,  
\end{equation}
where $\kappa: [0,\epsilon] \mapsto \mathbb{R}$ is  concave differentiable function satisfying $\kappa(0) = 0$ and $\kappa'>0$.
\end{definition}    
This definition is satisfied by a large class of functions, for example, functions that are semialgebric (e.g., real polynomial functions), vector or matrix (semi)norms (e.g., $\|\cdot\|_p$ with $p \geq 0$ rational number), see \cite{BolDan:07} for a comprehensive list.


\section{A moving Taylor approximation method}\label{sec:3}

In this section, we introduce a new higher-order algorithm, which we call  \textit{Moving Taylor Approximation} (MTA) algorithm, for solving the constrained optimization problem \eqref{eq:problem}. We denote the  feasible set of \eqref{eq:problem} by ${\cal F} = \{x\in\mathbb{E}:   F_i(x) \leq 0 \;\; \forall i=1:m\}$. In this paper, we consider the following assumptions for the objective and the constraints:
\begin{assumption}
\label{ass:1}
We have the following assumptions for problem \eqref{eq:problem}:
\begin{enumerate}
    \item The (possibly nonconvex) function $F_0$ is $p$-times continuously differentiable (with $p \geq 1$) and its $p$th derivative satisfy the Lipschitz condition:
\begin{align*}
 \|  D^p F_0(x)  - D^p F_0(y) \| \leq L_p  \|x-y\| \quad \forall x,y\in\mathbb{E}. 
\end{align*}
\item The (possibly nonconvex) constraints $F_i$ are  $q$-times continuously differentiable (with $q \geq 1$) and their $q$th derivatives satisfy the Lipschitz condition:
\[  \|  D F_i^q(x)  - D F_i^q(y) \| \leq L_q^i \|x-y\| \quad \forall x,y\in\mathbb{E},  \; i=1:m.  \] 
\item $h$ is simple, proper, convex and locally Lipschitz continuous function. 
\end{enumerate}
\end{assumption}
Next, we assume that our problem is feasible and has bounded level sets:
\begin{assumption}
\label{ass:2}
Problem \eqref{eq:problem} is feasible, i.e., $\mathcal{F} \neq \emptyset$ and the set  $\mathcal{A}(x_0) := \{x \in \mathbb{E}: x \in {\cal F} \; \text{and} \; F(x) \leq F(x_0) \}$ is bounded for any fixed $x_0 \in {\cal F}$.  
\end{assumption}  
Finally, we assume that the Mangasarian-Fromovitch constraint qualification (MFCQ) holds for the problem \eqref{eq:problem}: 
\begin{assumption}
\label{ass:3}
 The MFCQ holds for the optimization problem \eqref{eq:problem}:
\[  \forall x \in {\cal F} \;\; \exists d \in \mathbb{E} \;\; \text{s.t.} \;\; \langle \nabla F_i(x), d \rangle <0 \;\; \forall i  \in I(x), \]
where $I(x):=\{i \in [m], \; F_i(x) =0 \}.$
\end{assumption} 

\noindent Note that Assumptions \ref{ass:2} and \ref{ass:3} are standard in the context of nonlinear programming.  In particular, the  MFCQ guarantees the existence of bounded Lagrange multipliers satisfying the KKT optimality conditions. Note that in general, for an optimization algorithm,  if one wants to  prove only local convergence rates around a local minimum $x^*$, then it is sufficient to require MFCQ  only at that point  $x^*$. However, if one wants to prove global convergence for an algorithm, it is needed to require MFCQ  on a set where all the iterates lie (see  Assumption \ref{ass:3} and e.g., also papers \cite{Nes:07,Lu:22}).  From Assumption \ref{ass:1}, we have for all $x,y \in \mathbb{E}$ \cite{Nes:20}:
\begin{align}
\label{eq:fineq}
& |F_0(y) - T_p^{F_0}(y;x)| \leq  \frac{L_p}{(p+1)!} \|y - x \|^{p+1},  
\end{align}
\begin{align}
\label{eq:fineq_c}
& |F_i(y) - T_q^{F_i}(y;x)| \leq  \frac{L_q^i}{(q+1)!} \|y - x \|^{q+1}, \; i=1:m,  
\end{align}

\noindent At each iteration our algorithm constructs  Taylor approximations for the objective function and the functional constraints using the inequalities given in \eqref{eq:fineq} and \eqref{eq:fineq_c}. To this end, for simplicity, we consider the following notations:
\begin{align*}
&s_{M_p}^M(y;x) \stackrel{\text{def}}{=} T_p^{F_0}(y;x) + \frac{M_p}{(p+1)!} \|y - x \|^{p+1} + \frac{M}{(q+1)!} \|  y - x \| ^{q+1},\\
&s_{M_{q}^i}(y;x) \stackrel{\text{def}}{=} T_q^{F_i}(y;x) + \frac{M_q^i}{(q+1)!} \|y - x \|^{q+1},
\end{align*}
\noindent where $M_p$, $M$ and $M_q^i$, for $i=1:m$, are positive constants. The MTA algorithm is defined in the table.


\begin{algorithm}
\caption{\hspace{1cm} \textbf{MTA}: Moving Taylor approximation}
\begin{algorithmic}\label{alg:1}
\STATE{Given $x_0 \in {\cal F}$ and \red{$M_p > L_p, M>0, M_q^i > L_q^i$}, for $i=1:m$,  and $k=0$.}
\WHILE{stopping criteria}
\STATE{compute $x_{k+1}$ an approximate stationary point of the subproblem:}
\begin{align}\label{eq:rn}
&\min_{x\in\mathbb{E}}  s_{M_p}^M(x;x_k) + h(x)\\
&\;\; \text{s.t.}: \; s_{M_q^i}(x;x_k)  \leq 0, \;\;  i=1:m, \nonumber 
\end{align}
satisfying the following descent:
\begin{align}
\label{eq:desc}
s_{M_p}^M(x_{k+1};x_k) + h(x_{k+1}) \leq s_{M_p}^M(x_k;x_k) + h(x_k) \; (:= F(x_k)).
\end{align}
\STATE{update $k=k+1$.}
\ENDWHILE
\end{algorithmic}
\end{algorithm}


\medskip 

\noindent \red{Note that our algorithm requires starting from an initial feasible point (i.e., $x_0 \in {\cal F}$), hence, problem \eqref{eq:problem} should permit computation of such a point.} Moreover, one can notice that  if $F_i$'s, for $i=0:m$, are convex functions, then the subproblem \eqref{eq:rn} is also convex. Indeed, if $M_p \geq pL_p$ and $M_q^i \geq qL_q^i$ for $i=1:m$, then the Taylor approximations $s_{M_p}^M(\cdot;x_k)$ and $s_{M_q^i}(\cdot;x_k)$ for $i=1:m$ are (uniformly) convex functions (see Theorem 2 in \cite{Nes:20}). Hence, in the convex case, by approximate stationary point in algorithm MTA we mean that $x_{k+1}$ is the global optimum of the convex subproblem \eqref{eq:rn}. However, in the nonconvex case, we cannot always guarantee the convexity of the subproblem \eqref{eq:rn}. In this case, by approximate stationary point in algorithm MTA we mean that $x_{k+1}$ satisfies Assumption \ref{th:kkt} below together with   the descent \eqref{eq:desc}. Moreover, in Section \ref{sec:4} we  show that one can still use the powerful tools from convex optimization for solving subproblem \eqref{eq:rn} even in the nonconvex case.  Note that our novelty comes from using  two regularization terms in the objective function of the subproblem \eqref{eq:rn}, i.e.: \\
(i) $\frac{M_p}{(p+1)!}\|x - x_k\|^{p+1}$ ensures the convexity of the subproblem in the convex case (provided that $M_p\geq pL_p$),\\
(ii) while $\frac{M}{(q+1)!}\|x - x_k\|^{q+1}$ ensures a descent for an appropriate Lyapunov function (see Lemma \ref{lem:4}) and a better convergence rate (see Remark  \ref{rem:1}). 

\medskip 

\noindent We denote the feasible set of subproblem \eqref{eq:rn} by: $${\cal F}(x_k): = \{ y\in\mathbb{E}:\;s_{M_q^i}(y;x_k) \leq 0 \;\; \forall i=1:m \}.$$

\section{Nonconvex convergence analysis}
\label{sec:nonconvex}
In this section we assume that the constrained optimization problem \eqref{eq:problem} is nonconvex, i.e., $F_i$'s, for $i=0:m$, are (possibly) nonconvex functions. Then, the subproblem \eqref{eq:rn} is usually also nonconvex. Consequently, in this section we assume that $x_{k+1}$ satisfies  the descent \eqref{eq:desc}, and, additionally, is an \textit{approximate stationary point} of the subproblem \eqref{eq:rn} as defined in the following assumption: 
\red{
\begin{assumption}
    \label{th:kkt}
    Given $\eta_1,\eta_2,\eta_3 > 0$, there exist Lagrange multipliers $u^{k+1} = (u_1^{k+1},\cdots,u_m^{k+1}) \geq 0$ and $\Lambda_{k+1}\in\partial h(x_{k+1})$ such that for all $i=1:m$, the following approximate KKT (A-KKT) conditions hold for $x_{k+1}$ corresponding  to the subproblem \eqref{eq:rn} in  MTA algorithm:
    \begin{align}\label{eq:CAKKT}
       &\left\|\nabla s_{M_p}^M(x_{k+1};x_k) + \Lambda_{k+1} + \sum_{i=1}^{m}u_i^{k+1}\nabla s_{M_q^i}(x_{k+1};x_k)\right\|\leq \eta_1\|x_{k+1} - x_k\|^{\min(p,q)},\nonumber \\
       &|u_i^{k+1}s_{M_q^i}(x_{k+1};x_k) |\leq \frac{\eta_2}{(q+1)!}\|x_{k+1} \!- x_k\|^{q+1} \quad \forall i=1:m, \\ 
       &\left(\!s_{M_q^i}(x_{k+1};x_k) \!\right)_+ \!\! \leq \frac{\eta_3}{(q\!+\!1)!} \|x_{k+1} \!- x_k\|^{q+1} \quad  \forall i=1:m, \;  \text{ where } (a)_{+} =\max(0,a). \nonumber
    \end{align}    
\end{assumption}}
\noindent  Similar approximate KKT conditions have been studied in \cite{AnMaSaSa:14, AnMaSv:10,  Mar:17}. \red{However, our conditions are more general as we allow the right hand side terms  to be dynamic}. Note that in \cite{AnMaSaSa:14} (page 3) it has been shown that at any local minimizer of subproblem \eqref{eq:rn} there exist $x_{k+1}$ and $(u_i^{k+1})_{i=1}^{m} \geq 0$ such that A-KKT conditions \eqref{eq:CAKKT} hold.  In general, if the original problem \eqref{eq:problem} satisfies some constraint qualifications (e.g., MFCQ) on $\cal F$, then the corresponding subproblem \eqref{eq:rn} may satisfy some constraint qualifications as well, which however are not necessarily of the same type as that of the original problem. For example, if the original problem satisfies MFCQ and the subproblem \eqref{eq:rn} is convex, then the Slater condition holds for the subproblem and, consequently, Assumption \ref{th:kkt} holds with $\eta_1=\eta_2 = \eta_3 = 0$ at $x_{k+1}$ (the global minimum of the convex subproblem). Indeed, let us prove that the Slater condition holds under MFCQ (Assumption \ref{ass:3}), i.e., for any $x \in {\cal F}$ fixed, there exists a $\zeta \in \mathbb{E}$ such that the inequality constraints in \eqref{eq:rn} hold strictly, i.e.:
\begin{align*}
F_i(x) + \langle \nabla F_i(x), \zeta - x \rangle + \sum_{j=2}^{q}\frac{1}{j!}\nabla^j F_i (x)[\zeta - x]^{j}  +  \frac{M_q^i}{(q+1)!} \| \zeta  - x \|^{q+1} < 0,
\end{align*}
for all $i=1:m$. Using a similar argument as in \cite{AuShTe:10}, we show that  a point of the form $\zeta = x + t d$, with $t \in \mathbb{R}_+$ and $d \in \mathbb{E}$ such that $\|d\| =1$,  satisfies strictly these inequalities provided that $t$ is sufficiently small.  Indeed, for an inactive constraint   $F_i(x) <0$, we have:
\begin{align*}
& F_i(x) + \langle \nabla F_i(x), \zeta - x \rangle + \sum_{j=2}^{q}\frac{1}{j!}\nabla^j F_i (x)[\zeta - x]^{j}  +  \frac{M_q^i}{(q+1)!} \| \zeta  - x \|^{q+1}\\ 
&= F_i(x) + t\langle \nabla F_i(x), d \rangle + \sum_{j=2}^{q} t^j\frac{1}{j!}\nabla^j F_i (x)[d]^{j}  +  t^{q+1}\frac{M_q^i}{(q+1)!} \| d \|^{q+1}\\ 
&\leq F_i(x) + t \|\nabla F_i(x)\| + \sum_{j=2}^{q} t^j\frac{1}{j!} \|\nabla^j F_i (x)\| + t^{q+1}\frac{M_q^i}{(q+1)!} <0,
\end{align*}
\noindent where the first inequality follows from Cauchy-Schwartz  and the last inequality  from $F_i(x) <0$ and $t$  is sufficiently small. For an active constraint $F_i(x) = 0$, from Assumption \ref{ass:3} we have $\langle \nabla F_i(x), d \rangle <0$ for some $d$. Hence, using a similar argument as above, we have: \begin{align*}
&F_i(x) + \langle \nabla F_i(x), \zeta - x \rangle + \sum_{j=2}^{q}\frac{1}{j!}\nabla^j F_i (x)[\zeta - x]^{j}  +  \frac{M_q^i}{(q+1)!} \| \zeta  - x \|^{q+1}\\ 
& =  t\langle \nabla F_i(x), d \rangle + \sum_{j=2}^{q} t^j\frac{1}{j!} \|\nabla^j F_i (x)\| + t^{q+1}\frac{M_q^i}{(q+1)!} <0,
\end{align*}
provided that $t$ is sufficiently small. This shows that the Taylor approximation inequality constraints in   \eqref{eq:rn} have a nonempty interior, and thus the Slater condition holds for the subproblem \eqref{eq:rn}.

\medskip 

\noindent Next, we show that the sequence $(F(x_k))_{k\geq 0}$ is strictly nonincreasing.
\begin{lemma}\label{th:1}
Let Assumptions \ref{ass:1}, \ref{ass:2}, \ref{ass:3}, and \ref{th:kkt} hold and the sequence $(x_k)_{k\geq 0}$ be generated by MTA algorithm with $x_0 \in \cal F$, $M_p>L_p, M>0$ and $M_q^i \geq L_q^i + \red{\eta_3}$  for all $i=1:m$. Then, we have:
\begin{itemize}

\item[(i)]  The sequence $(F(x_k))_{k\geq 0}$ is nonincreasing and satisfies the descent:
\begin{align*}
F(x_{k+1}) \leq F(x_k) - \left( \frac{M_p \!-\! L_p}{(p+1)!} \|x_{k+1} - x_k \|^{p+1} + \frac{M}{(q+\!1)!} \|x_{k+\!1} - x_k \|^{q+\!1}  \right). 
\end{align*}

\item[(ii)] The set ${\cal F}(x_k)$ is nonempty and ${\cal F}(x_k) \subseteq {\cal F}$ for all $k\geq 0$. Additionally, the sequence $(x_k)_{k\geq 0}$ is feasible for the original problem \eqref{eq:problem}, bounded, and has a finite length, i.e.:
\begin{align*}
    \sum_{k=0}^{\infty}\left(\|x_{k+1} - x_k\|^{q+1} + \|x_{k+1} - x_k\|^{p+1}\right) <\infty.
\end{align*}
\end{itemize}
\end{lemma} 

\begin{proof}
(i) Writing  inequality \eqref{eq:desc} explicitly,  we have:
\begin{align*}
T_p^{F_0}(x_{k+1};x_k) + \frac{M_p}{(p+1)!}\| x_{k+1} \!\!-\! x_k \|^{p+1} + \frac{M}{(q+1)!}\|x_{k+1} \!\!-\! x_k \|^{q+1}  + h(x_{k+1}) \leq F(x_k).
\end{align*}


\noindent On the other hand, from \eqref{eq:fineq} we have:
\begin{align*}
    -\frac{L_p}{(p+1)!}\|x_{k+1} - x_k\|^{p+1} + F_0(x_{k+1})\leq T_p^{F_0}(x_{k+1};x_k), 
\end{align*}
which, combined with the previous inequality, yields:
\begin{align*}
\frac{M_p - L_p}{(p+1)!}\|x_{k+1} - x_k\|^{p+1} + \frac{M}{(q+1)!}\|x_{k+1} - x_k \|^{q+1} + F(x_{k+1})\leq F(x_k),
\end{align*}
proving the first statement of the lemma. \\

\noindent (ii) Regarding the second statement, let us first prove that ${\cal F}(x_k) \subseteq {\cal F}$ for all $k\geq 0$. Indeed, if $y \in {\cal F}(x_k)$, then $s_{M_q^i}(y;x_k) \leq 0$ for all $i=1:m$. Then, from \eqref{eq:fineq_c} and the fact that  $M_q^i \geq L_q^i$, we have:
\[  F_i(y) \leq  s_{L_q^i}(y;x_k) \leq s_{M_q^i}(y;x_k) \leq 0 \quad \forall i=1:m,  \]
which implies  that $y \in {\cal F}$ and thus ${\cal F}(x_k) \subseteq {\cal F}$.  
Further, if $M_q^i \geq L_q^i + \eta_3$ and $x_0\in \cal F$, then the subproblem \eqref{eq:rn} is well-defined, i.e., its feasible set is nonempty,  ${\cal F}(x_k) \neq \emptyset$, and, additionally,  $x_k$ is  also feasible for the original problem \eqref{eq:problem} for all $k\geq 0$. Indeed,  ${\cal F}(x_0) \neq \emptyset$, since $x_0 \in {\cal F}(x_0)$ (recall that  $x_0\in {\cal F}$, hence we have $s_{M_q^i}(x_0;x_0) = F_i(x_0) \leq 0$ for all $i=1:m$ and thus $x_0 \in {\cal F}(x_0)$).  Now let us prove that $x_1$ is feasible for problem \eqref{eq:problem}. Indeed,  from Assumption  \ref{th:kkt} and  \eqref{eq:fineq_c}, we~get:
\begin{align*}
 \frac{M_q^i - L_q^i}{(q+1)!}\|x_1 - x_0\|^{q+1} + F_i(x_1) & \leq s_{M_q^i}(x_1;x_0) \leq {\left(s_{M_q^i}(x_1;x_0) \right)_+}\\
 &\leq \frac{\red{\eta_3}}{(q+1)!} \|x_1 - x_0\|^{q+1}\;\;\; \forall i=1:m,
\end{align*}
and thus $F_i(x_1) \leq 0$  provided that  $M_q^i \geq L_q^i + \red{\eta_3}$ for all $i=1:m$. Hence,   $x_1 \in \mathcal{F}$. Therefore, the iterate $x_{1}$ is  feasible for the original problem \eqref{eq:problem}. Further, since $x_1 \in \mathcal{F}$, using the same arguments as for $x_0$, we get that $x_1 \in \mathcal{F}(x_1)$ and consequently ${\cal F}(x_1) \neq \emptyset$.  By induction, using the same arguments as before, we can easily prove that $ x_k \in {\cal F}(x_k) \neq \emptyset$  and  $x_k \in \mathcal{F}$   (i.e., $x_k$ is  feasible for the original problem \eqref{eq:problem}) for all $k\geq 0$. Further, since $(F(x_k))_{k\geq 0}$ is nonincreasing, then $x_{k}\in\mathcal{A}(x_0)$ and hence from Assumption \ref{ass:2} the sequence $(x_k)_{k\geq 0}$ is bounded. Finally, the last statement follows by summing up the descent inequality in function values from item (i).
\end{proof}

\noindent Note that from the  previous lemma, it follows that there exists $D > 0$ such that $$\|x_k\|\leq D \quad  \forall k\geq 0.$$ 
Therefore, the sequence $(x_k)_{k\geq 0}$ has limit points. Next, we show that the sequence of the multipliers $(u^{k})_{k\geq 1}$ given in \eqref{eq:CAKKT} is bounded.

\begin{lemma}\label{lem:3}
Let Assumptions \ref{ass:1}, \ref{ass:2}, \ref{ass:3} and \ref{th:kkt} hold. Then,  the multipliers $(u^{k})_{k\geq 0}$ defined in \eqref{eq:CAKKT} are bounded, i.e., there exists $C_u>0$ such that:
    \begin{align*}
        \|u^{k}\|\leq C_u \;\; \forall k\geq 0.
    \end{align*}
\end{lemma}

\begin{proof}
We use a similar reasoning as in \cite{AuShTe:10,YuPoLu:21}. Suppose the contrary, that is, there exist $u_i \geq 0$, for $i=1:m$, $x_{\infty} \in\mathcal{F}$, a subsequence $(u^k)_{k\in K}$, $(x_k)_{k\in K}$ with $K\subseteq \mathbb{N}$ and $\Lambda_{k+1}\in\partial h(x_{k+1})$ such that:
\begin{align*}
     \lim_{k\mapsto \infty,k\in K} \sum_{i=1}^{m} u_i^{k+1} = +\infty,\;  \lim_{k\mapsto \infty,k\in K} \frac{u_i^{k+1}}{\sum_{i=1}^{m} u_i^{k+1}} = u_i,\; \text{with} \sum_{i=1}^{m}u_i = 1,
\end{align*}
and additionally $\lim \limits_{k\mapsto \infty, k\in K} x_{k+1} = x_\infty$ and $\lim \limits_{k\mapsto \infty, k\in K}\Lambda_{k+1} = \Lambda \in\partial h(x_\infty)$ (note that $\partial h(x)$ is closed and bounded for all $x\in\text{dom}\, h$, since $h$ is assumed to be locally Lipschitz, see Theorem 9.13 in \cite{Roc:70}). Further, dividing the first and second equalities in \eqref{eq:CAKKT}  with the quantity $\sum_{i=1}^{m} u_i^{k+1}$, we get:
\begin{align*}
     &\frac{1}{\sum_{i=1}^{m} u_i^{k+1}}\left\| \nabla s_{M_p}^M(x_{k+1};x_k) + \Lambda_{k+1} + \sum_{i=1}^{m} u_i^{k+1} \nabla s_{M_q^i}(x_{k+1};x_k)\right\|\\
     &\qquad\qquad\qquad \leq  \frac{\red{\eta_1}}{\sum_{i=1}^{m} u_i^{k+1}}\|x_{k+1} - x_k\|^{\min(p,q)},\\
     &\frac{u_i^{k+1}}{\sum_{i=1}^{m} u_i^{k+1}}\left |s_{M_q^i}(x_{k+1};x_k)\right | \leq \frac{\red{\eta_2}}{\sum_{i=1}^{m} u_i^{k+1}(q+1)!} \|x_{k+1} - x_k\|^{q+1}.
\end{align*}
\noindent Since the Taylor functions $s_{M_p}^M$ and $s_{M_q^i}$'s, for $i=1:m$, are continuous and $(x_k)_{k\geq 0}$ is bounded, then passing to the limit as $k \to\infty, \;k \in K$, we obtain:
\begin{align*}
    &\sum_{i=1}^{m}u_i\nabla s_{M_q^i}(x_{\infty};x_{\infty}) = 0,
    & u_i s_{M_q^i}(x_{\infty};x_{\infty}) = 0,\; s_{M_q^i}(x_\infty;x_\infty)\leq 0\;\;\text{for}\; i=1:m.    
\end{align*}
\noindent It follows from the definition of $s_{M_q^i}$, for $i=1:m$,  that:
\begin{align*}
    &\sum_{i=1}^{m} u_i \nabla F_i(x_\infty) = 0, \;\;u_i F_i(x_{\infty}) = 0,\;\; F_i(x_\infty)\leq 0\;\;\text{for}\; i=1:m. 
\end{align*}
\noindent If $I(x_\infty) = \emptyset$ (see Assumption \ref{ass:3}), then for all $i=1:m$, $F_i(x_\infty) <0$ and hence $u_i = 0$, for $i=1:m$. This is a contradiction with $\sum_{i=1}^{m}u_i = 1$. Further, assume that $ I(x_\infty) \neq \emptyset$. Since we have $\sum_{i=1}^{m}u_i = 1$, then there exists $\mathcal{I}\subseteq I(x_\infty)$, $\mathcal{I} \not= \emptyset$, such that $u_i  > 0$ for all $i\in\mathcal{I}$.  From Assumption \ref{ass:3}, there exists $d\in\mathbb{E}$ such that:
\begin{align*}
    0 =  \left\langle\sum_{i=1}^{m} u_i\nabla  F_i(x_\infty),d \right\rangle =  \sum_{i\in \mathcal{I}}u_i \langle \nabla F_i(x_\infty),d\rangle <0,  
\end{align*}
which is a contradiction with MFCQ assumption. Hence, our statement follows.
\end{proof}


\subsection{Convergence rate to KKT points}\label{sec:kkt}
\noindent In the general nonconvex case we want to see how fast we can satisfy (approximately) the KKT  optimality conditions for the problem \eqref{eq:problem}. We consider points satisfying the first order local necessary optimality conditions for problem \eqref{eq:problem}, i.e., points which belong to ${\cal S}$:
\begin{align}\label{eq:kkt-op2}
 {\cal S} = &\left\lbrace x\in {\cal F}: \exists\; u_i \geq 0,\; \Lambda\in \partial h(x) \;\text{s.t.}:\nabla F_0(x) + \Lambda \!+\! \sum_{i=1}^{m} u_i \nabla F_i(x) \!=\! 0,\right.\\ 
 &\qquad\qquad\qquad\qquad\qquad \left.u_i F_{i}(x) \!=\! 0, \;i=1\!:\!m \right\rbrace.\nonumber   
\end{align}

\noindent Hence, an appropriate measure of optimality is optimality and complementary violations of KKT conditions. Therefore, for $\Lambda_{k+1}\in\partial h(x_{k+1})$ we define the  map:
\begin{align*}
{\cal M}(x_{k+1}) =\max &\left\{ \left\|  \nabla F_0(x_{k+1}) + \Lambda_{k+1} + \sum_{i=1}^m u^{k+1}_i \nabla F_i(x_{k+1}) \right\|,\right.\\
&\qquad\left. \Big(- u^{k+1}_i F_i(x_{k+1})\Big)^{\frac{q}{q+1}}, \; i=1:m \right\}.
\end{align*}
 \noindent Recall that   $M_p > L_p$.  For simplicity, let us introduce the following constants $C_1 = \frac{L_p + M_p}{p!} $, $C_2 = \left(\frac{C_u\sum_{i=1}^{m} (M_q^i + L_q^i) + M}{q!} +  \left(C_u\left( \max_{i=1:m} \frac{M_q^i + L_q^i }{(q+1)!}\right) + \frac{\red{\eta_2}}{(q+1)!}\right)^{\frac{q}{q+1}} \right) $ and 
 
\begin{align*}
 C = \max&\left(\frac{((q+1)!)^{\frac{q}{q+1}}\left(C_1 (2D)^{p-q} + \red{\eta_1} + C_2 \right) (F(x_0) - F_{\infty})^{\frac{q}{q+1}} }{M^{\frac{q}{q+1}}},\right.\\
&\qquad\qquad\qquad\left. \frac{((p+1)!)^{\frac{p}{p+1}}\left(C_1 + \red{\eta_1}  + C_2 (2D)^{q-p} \right) (F(x_0) - F_{\infty})^{\frac{p}{p+1}} }{(M_p - L_p)^{\frac{p}{p+1}}}  \right).
\end{align*}

\noindent Then, we have the following convergence rate for the measure of optimality $\mathcal{M}(x_k)$:

\begin{theorem}\label{th:2}
Let the assumptions of Lemma \ref{th:1} hold and $(x_k)_{k\geq 0}$ be generated by MTA algorithm. Then, there exists $\Lambda_{k+1}\in\partial h(x_{k+1})$ such that:
\begin{enumerate}
    \item[(i)]  The following bound holds:
    \begin{align*}
         {\cal M}(x_{k+1})\leq C_1\|x_{k+1} - x_k\|^p +\red{\eta_1 \|x_{k+1}- x_k\|^{\min(p,q)}} + C_2\|x_{k+1} - x_k\|^q.   
    \end{align*}
    \red{Consequently, if $x_{k+1} = x_k$, then  $x_k$ is a KKT point of the original nonconvex problem \eqref{eq:problem}. }
    
    \item[(ii)] The sequence $\left(\mathcal{M}(x_k)\right)_{k\geq 0}$ converges to $0$ with the following sublinear rate:
    
    \begin{align*}
      \min_{j=1:k}{\cal M}(x_j)\leq \frac{C}{k^{\min \left(\frac{q}{q+1},\frac{p}{p+1}\right)}}.   
    \end{align*}
\end{enumerate}
\end{theorem}

\begin{proof}
Let $u^{k+1}_i \geq 0$  be given in \eqref{eq:CAKKT}, then from \eqref{eq:fineq_c}  we get for all $i=1:m$ 

\begin{align*}
&- u^{k+1}_i \frac{L_q^i}{(q+1)!} \|x_{k+1} - x_k \|^{q+1}  \leq u^{k+1}_i \Big( F_i(x_{k+1}) - T^{F_i}_q (x_{k+1} ; x_k)\Big) 
\\
&\qquad\qquad = u^{k+1}_i  F_i(x_{k+1})  - u^{k+1}_i \left(T^{F_i}_q (x_{k+1} ; x_k) + \frac{M_q^i}{(q+1)!}\|x_{k+1} - x_k\|^{q+1}  \right) \\
&\quad\qquad\qquad + u^{k+1}_i  \frac{M_q^i}{(q+1)!}\|x_{k+1} - x_k\|^{q+1} \\
&\qquad\qquad =  u^{k+1}_i  F_i(x_{k+1}) - u^{k+1}_i s_{M_q^i}(x_{k+1};x_k) + u^{k+1}_i  \frac{M_q^i}{(q+1)!} \|x_{k+1} - x_k \|^{q+1}\\
&\qquad\qquad \leq u^{k+1}_i  F_i(x_{k+1}) + \red{ \frac{\eta_2}{(q+1)!} \|x_{k+1} - x_k \|^{q+1}} + u^{k+1}_i  \frac{M_q^i}{(q+1)!} \|x_{k+1} - x_k \|^{q+1},
\end{align*}
\red{where the the last inequality follows from \eqref{eq:CAKKT}}. Since the multipliers are bounded, taking the maximum, we get:

\begin{align}\label{eq:1}
&\max_{i=1:m} \left\lbrace \left(- u^{k+1}_i  F_i(x_{k+1})\right)^{\frac{q}{q+1}} \right\rbrace \\ \nonumber
&\qquad \leq \max_{i=1:m} \left\lbrace \left(u^{k+1}_i \left(\frac{M_q^i + L_q^i}{(q+1)!}\right) + \frac{\red{\eta_2}}{(q+1)!} \right)^{\frac{q}{q+1}} \|x_{k+1} - x_k \|^q \right\rbrace\\ \nonumber 
&\qquad\leq \left(C_u \left( \max_{i=1:m} \frac{M_q^i + L_q^i }{(q+1)!} \right) + \frac{\red{\eta_2}}{(q+1)!}\right)^{\frac{q}{q+1}} \|x_{k+1} - x_k \|^{q}.
\end{align}
Further, let $\Lambda_{k+1}\in\partial h(x_{k+1})$, then we have:
\medskip
\begin{align*}
     &\left\| \nabla F_0(x_{k+1}) + \Lambda_{k+1} + \sum_{i=1}^m u^{k+1}_i \nabla F_i(x_{k+1}) \right\| \\
     &\leq  \|\nabla F_0(x_{k+1})  - \nabla T^{F_0}_p(x_{k+1} ; x_k) \| + \left\Vert \nabla T^{F_0}_p(x_{k+1};x_k) + \Lambda_{k+1}\right.\\
     &\qquad\left. + \frac{M_p}{p!}\|x_{k+1} - x_k\|^{p-1}(x_{k+1} - x_k) + \frac{M}{q!}\|x_{k+1} - x_k\|^{q-1}(x_{k+1} - x_k)\right.\\
    &\qquad \left. + \sum_{i=1}^m u^{k+1}_i \left(\nabla T^{F_i}_q(x_{k+1};x_k) + \frac{M_q^i}{q!}\|x_{k+1} - x_k\|^{q-1}(x_{k+1} - x_k)\right) \right\Vert\\
    & \quad + \left\Vert -\frac{M_p}{p!} \|x_{k+1} - x_k\|^{p-1}(x_{k+1} - x_k)\right.\\
    & \left. \qquad\quad - \frac{\sum_{i=1}^m u^{k+1}_i M_q^i + M}{q!}\|x_{k+1} - x_k\|^{q-1}(x_{k+1} - x_k)\right\Vert \\
    &\quad + \left\Vert \sum_{i=1}^m u^{k+1}_i \left(\nabla F_i(x_{k+1}) - \nabla T^{F_i}_q(x_{k+1};x_k)\right)  \right\Vert\\
    & \leq  \frac{L_p + M_p}{p!} \|x_{k+1} - x_k\|^p + \red{\eta_1 \|x_{k+1}- x_k\|^{\min(p,q)}} \\
    &\qquad\quad+ \frac{\sum_{i=1}^m u^{k+1}_i (M_q^i+L_q^i) + M}{q!} \|x_{k+1} - x_k\|^q \\
    &\leq \frac{L_p + M_p}{p!}\|x_{k+1} - x_k\|^p + \red{\eta_1 \|x_{k+1}- x_k\|^{\min(p,q)}} \\
    &\qquad\quad+ \frac{ C_u \sum_{i=1}^{m} (M_q^i + L_q^i) + M}{q!}\|x_{k+1} - x_k\|^q,
\end{align*}
where the second inequality follows from \eqref{eq:fineq}, \eqref{eq:fineq_c}, and Assumption \ref{th:kkt}, \red{i.e., the relations \eqref{eq:CAKKT}}. The last inequality follows from Lemma \ref{lem:3}. Combining this inequality with \eqref{eq:1}, we get:
\begin{align*}
    &{\cal M}(x_{k+1}) \leq \frac{L_p + M_p}{p!}\|x_{k+1} - x_k\|^{p} + \red{\eta_1 \|x_{k+1}- x_k\|^{\min(p,q)}}\\
    &\; + \!\left(\!\frac{C_u \sum_{i=1}^{m} (M_q^i + L_q^i) \!+\! M}{q!} + \! \left(C_u \! \left( \! \max_{i=1:m} \frac{M_q^i + L_q^i }{(q+1)!}\right) \!+ \frac{\red{\eta_2}}{(q+1)!}\right)^{\frac{q}{q+1}}\! \right) \! \|x_{k+1} - x_k\|^q\\
     &= C_1 \|x_{k+1} - x_k\|^p + \red{\eta_1 \|x_{k+1}- x_k\|^{\min(p,q)}} + C_2\|x_{k+1} - x_k\|^q.
\end{align*}
 Hence, the first assertion follows. Further, if $q\leq p$, then it follows that:
\begin{align*}
    {\cal M}(x_{k+1}) \leq \left(C_1 \|x_{k+1} - x_k\|^{p-q} + \red{\eta_1} + C_2 \right)\|x_{k+1} - x_k\|^q. 
\end{align*}
Since we have $\|x_{k+1} - x_k\|\leq 2D$ (see Lemma \ref{th:1}), then:
\begin{align*}
    \mathcal{M}(x_{k+1})\leq \left(C_1 (2D)^{p-q} + \red{\eta_1}+ C_2 \right)\|x_{k+1} - x_k\|^q.
\end{align*}
Combining this inequality with the descent from Lemma \ref{th:1} (i), we get:
\begin{align*}
    \mathcal{M}(x_{k+1})^{\frac{q+1}{q}}\leq \frac{\left(C_1 (2D)^{p-q} + \red{\eta_1} + C_2 \right)^{\frac{q+1}{q}} (q+1)!}{M}\big( F(x_k) - F(x_{k+1})\big).
\end{align*}
Summing up this inequality and taking the minimum, we obtain:
\begin{align}\label{eq:cv0}
\min_{j=1:k}{\cal M}(x_j)\leq \frac{((q+1)!)^{\frac{q}{q+1}} (C_1 (2D)^{p-q} + \red{\eta_1} + C_2)(F(x_0) - F_{\infty})^{\frac{q}{q+1}} }{M^{\frac{q}{q+1}}k^{\frac{q}{q+1}}}.
\end{align}
On the other hand, if $p\leq q$, then we also have:
    \begin{align*}
       {\cal M}(x_k) &\leq \left(C_1 + \red{\eta_1}  + C_2 \|x_{k+1} - x_k\|^{q-p} \right)\|x_{k+1} - x_k\|^p\\
             &\leq \left(C_1 + \red{\eta_1} + C_2 (2D)^{q-p} \right) \|x_{k+1} - x_k\|^p. 
    \end{align*}
Combining this inequality with the descent from Lemma \ref{th:1} (i), we get:
\begin{align*}
    \mathcal{M}(x_{k+1})^{\frac{p+1}{p}}\leq \frac{\left(C_1 + \red{\eta_1}  + C_2 (2D)^{q-p} \right)^{\frac{p+1}{p}} (p+1)!}{M_p - L_p} \left(F(x_k) - F(x_{k+1})\right).
\end{align*}
Summing up this inequality and taking the minimum, we obtain:
\begin{align}\label{eq:cv1}
\min_{j=1:k}{\cal M}(x_j)\leq \frac{((p+1)!)^{\frac{p}{p+1}} (C_1 + \red{\eta_1} + C_2 (2D)^{q-p})(F(x_0) - F_{\infty})^{\frac{p}{p+1}} }{(M_p - L_p)^{\frac{p}{p+1}}k^{\frac{p}{p+1}}}.
\end{align}
 Hence, combining inequalities \eqref{eq:cv0} and \eqref{eq:cv1}, our assertion follows.
\end{proof}

\begin{remark}\label{rem:1}
    Theorem \ref{th:2}, shows that  there exist a subsequence of the sequence $(x_k)_{k\geq 0}$, generated by MTA algorithm, which converges to a KKT point of the original problem \eqref{eq:problem}. \red{Moreover,  from Theorem \ref{th:2} (i) it follows that any fixed point of MTA algorithm is a KKT point of the original nonconvex problem \eqref{eq:problem}.}  If $p=q$, then the convergence rate is of order $\mathcal{O}\left(k^{-\frac{p}{p+1}}\right)$, which is the usual convergence rate for higher-order algorithms for (unconstrained) nonconvex problems  \cite{Mar:17,NecDan:20,CaGoTo:20, CaGoTo:22, NabNec:23}.   If $M = 0$ (in this case  \eqref{eq:cv0} is replaced with an inequality similar to the one in \eqref{eq:cv1}), then the convergence rate in the minimum of the optimality map $\mathcal{M}(x_k)$ is of order $\mathcal{O}\left(k^{-\min\left(\frac{q}{p+1}, \frac{p}{p+1}\right)}\right)$. Thus, if $q\leq p$ we have $\frac{q}{p+1}\leq \frac{p}{p+1}$, and hence the convergence rate becomes $\mathcal{O}\left(k^{-\frac{q}{p+1}}\right)$, which is worse than the rate $\mathcal{O}\left(k^{-\frac{q}{q+1}}\right)$. For a better understanding of this situation, consider a particular case: $p=2$ and $q=1$. Then, for $M=0$, the rate is $\mathcal{O} \left(k^{-\frac{1}{3}} \right)$, while for $M >0$ the rate is $\mathcal{O} \left(k^{-\frac{1}{2}} \right)$. In conclusion, it is beneficial to have additionally the regularization of order $q+1$ in the objective function since it leads to faster convergence rates. 
\end{remark}


\red{
\subsection{Adaptive MTA}
Note that  MTA algorithm \ref{alg:1}  requires knowledge of  the Lipschitz constants $L_p$ and $L_q^i$ for all $i=1:m$.  However, such parameters may not always be available. To address this limitation, we introduce below an adaptive MTA algorithm that does not require prior knowledge of these Lipschitz constants, making it  practical  in scenarios where such information is unavailable or difficult to estimate.}

\begin{algorithm}
\caption{Adaptive-MTA: Adaptive Moving Taylor Approximation}
\label{alg:adap1}
\begin{algorithmic}[1] 
\STATE \textbf{Given:} $x_0 \in {\cal F}$,  $R_p > 0$, $\eta_1,\eta_2,\eta_3 >0$, and $ M, M_{p,0}, M_{q,0}^i > 0$ for $i=1:m$.\\ 
Initialize: $j \gets 0$ and $k \gets 0$.
\WHILE{stopping criteria not met}
    \STATE \textbf{Step 1:} Compute $x_{k+1}$ for \eqref{eq:rn} satisfying \eqref{eq:desc} and \eqref{eq:CAKKT} with parameters: \\
    $
    M_{p,k} \gets  2^{j} M_{p,k} $ \; and \;  $ M_{q,k}^{i} \gets  2^{j} M_{q,k}^i \; \text{ for } i=1:m.
    $
    \IF{the following properties hold
    \begin{align} 
    \label{desc:adaptiveMTA}
        &\frac{R_p}{(p+1)!} \|x_{k+1} - x_k\|^{p+1} + \frac{M}{(q+1)!}\|x_{k+1} - x_k\|^{q+1} \leq F(x_k) - F(x_{k+1}) \\
        &\text{ and } \; F_i(x_{k+1}) \leq 0\; \text{ for all }\; i=1:m \nonumber
    \end{align}}
    \STATE \textbf{Step 2:} Update  $M_{p,k+1} \gets  \max(2^{j-1}M_{p,k}, M_{p,0})$ and $M_{q,k+1}^i \gets \max(2^{j-1}M_{q,k}^i , M_{q,0}^i) \;\;  \forall i=1:m$, \;\;  $k  \gets  k+1$ \; and \; $j \gets 0$.
    \ELSE 
    \STATE Set $j \gets j+1$ and go to \textbf{Step 1}. 
    \ENDIF
\ENDWHILE
\end{algorithmic}
\end{algorithm}

\medskip 

\noindent \red{First, one should notice that Algorithm  \ref{alg:adap1} is well-defined, i.e., at each iteration $k$, Step 1  is executed  only  a finite number of times (or,  equivalently, $j$ is updated only a finite number of steps). Indeed, as observed in the proof of \cref{th:1}, we can ensure the descent and the feasibility conditions in \eqref{desc:adaptiveMTA} if $M_{p,k}$ satisfies   \( R_p +  L_p \leq M_{p,k}  \) together with the descent  \eqref{eq:desc} and $ M_{q,k}^i$'s satisfy \( \eta_3 + L_q^i \leq M_{q,k}^i \) for all \( i = 1:m \), respectively. Hence, doubling $M_{p,k}$ and $ M_{q,k}^i$'s at each step $j$, these conditions will be met after a finite number of doublings. However, the conditions \eqref{desc:adaptiveMTA} may hold for values of $M_{p,k}$ and $M^i_{q,k}$ smaller than $L_p$ and $L^i_q$, respectively,  since these conditions need to hold only at $x_{k+1}$, not everywhere. }

\medskip 

\noindent \red{Note that the descent and the feasibility conditions   in \eqref{desc:adaptiveMTA} are sufficient to prove  the statements of \cref{th:1} and \Cref{lem:3}. Indeed, the descent condition in \eqref{desc:adaptiveMTA} coincides with  the first statement of  \cref{th:1}. Moreover,  the feasibility conditions in \eqref{desc:adaptiveMTA} allows us to show that the feasible set of subproblem \eqref{eq:rn} is nonempty,  obtaining  the second statement of   \cref{th:1}. Finally, computing $x_{k+1}$ satisfying  \eqref{eq:CAKKT}, leads to validity of Lemma  \ref{lem:3}.  Note also that we have $M_{p,0} \leq M_{p,k} \leq 2(R_p + L_p)$ and $M_{q,0}^i \leq M_{q,k}^i \leq 2(\eta_3 + L_q^i)$ for all $k\geq 0$ and $i=1:m$. Consequently,  following the same convergence analysis as in Section \ref{sec:kkt}, we can show that the limit points of the iterates generated by  Adaptive-MTA (Algorithm \ref{alg:adap1}) also converge to a KKT point of the  original problem~\eqref{eq:problem} and with the same convergence rate as in \cref{th:2}, where $M_p - L_p$ is replaced by $R_p$.}



\subsection{Better convergence under KL}\label{sub_kl}
In this section, we derive  convergence rates for our algorithm  under the KL property. To this end, consider the following Lagrangian function for the problem \eqref{eq:problem} and for the subproblem given in \eqref{eq:rn}:

\vspace{-0.5cm}

\begin{align*}
\mathcal{L}_p(x;u) {=} F(x) + \sum_{i=1}^{m} u_i F_i(x),\;\;\; 
\mathcal{L}_{sp}(y;x;u) {=} s_{M_p}^M(y;x) + h(x) + \sum_{i=1}^{m} u_is_{M_q^i}(y;x).
\end{align*}


\noindent Next, we establish the following results, known as descent-ascent \cite{BoSaTe:18}:
\begin{lemma}\label{th:3}
Let the assumptions of Theorem \ref{th:2} hold. Then, we have:

\vspace{-0.6cm}

    \begin{align}\label{eq:th3}
    &\mathcal{L}_p(x_{k+1};u^{k+1}) - \mathcal{L}_p(x_k;u^k) \nonumber  \\ 
    & \leq - \frac{(M_p - L_p)}{(p+1)!}\|x_{k+1} - x_k\|^{p+1} + \frac{C_u \|M_q + L_q\| + \red{m \eta_2}}{(q+1)!}\|x_k - x_{k-1}\|^{q+1} \nonumber \\  
      &\qquad\qquad\quad- \left( \frac{M + \sum_{i=1}^{m} u_i^{k+1}(M_q^i - L_q^i) + \red{m\eta_2}}{(q+1)!}\right) \|x_{k+1} - x_k\|^{q+1}.  
    \end{align}
\end{lemma}

\begin{proof}
We have:
\begin{align*}
 &\frac{M_p - L_p}{(p+1)!}\|x_{k+1} - x_k\|^{p+1} + \frac{M}{(q+1)!}\|x_{k+1} - x_k\|^{q+1} \leq F(x_k) - F(x_{k+1})\\
 &\leq  F(x_k) - F(x_{k+1}) - \sum_{i=1}^{m}u^{k+1}_i s_{M_q^i}(x_{k+1};x_k) + \red{\sum_{i=1}^{m} \frac{\eta_2}{(q+1)!} \|x_{k+1} - x_k\|^{q+1}} \\
 &\leq  F(x_k) - F(x_{k+1}) - \sum_{i=1}^{m}u^{k+1}_i\left(F_i(x_{k+1}) +\frac{M_q^i - L_q^i }{(q+1)!} \|x_{k+1} - x_k\|^{q+1}  \right)\\
 &\qquad + \red{\sum_{i=1}^{m} \frac{\eta_2}{(q+1)!} \|x_{k+1} - x_k\|^{q+1}} \\
 & = \mathcal{L}_p(x_k;u^{k+1}) -  \mathcal{L}_p(x_{k+1};u^{k+1}) -\sum_{i=1}^{m}u^{k+1}_i F_i(x_{k}) \\
 &\qquad - \sum_{i=1}^{m}u^{k+1}_i \left( \frac{M_q^i - L_q^i}{(q + 1)!} \|x_{k+1} - x_k\|^{q+1}  \right) + \red{\frac{m\eta_2}{(q+1)!} \|x_{k+1} - x_k\|^{q+1}},   
\end{align*}
where the first inequality follows from Lemma \ref{th:1} (i), \red{the second inequality follows from the A-KKT conditions \eqref{eq:CAKKT} (i.e., Assumption \ref{th:kkt})}, the third inequality follows from \eqref{eq:TayAppBound} and $u_i^k\geq 0$. The last equality follows from the definition of $\mathcal{L}_p$.  Furthermore, we have:

\vspace{-0.6cm}

\begin{align*}
 &\mathcal{L}_p(x_k;u^{k+1}) -  \mathcal{L}_p(x_{k+1};u^{k+1}) - \sum_{i=1}^{m}u^{k+1}_i F_i(x_{k}) \\
    & = \mathcal{L}_p(x_k;u^k) - \mathcal{L}_p(x_{k+1};u^{k+1}) + \mathcal{L}_p(x_k;u^{k+1}) - \mathcal{L}_p(x_k;u^k) - \sum_{i=1}^{m}u^{k+1}_i F_i(x_{k})\\
    & = \mathcal{L}_p(x_k;u^k) - \mathcal{L}_p(x_{k+1};u^{k+1}) + \sum_{i=1}^{m}u_i^{k+1}F_i(x_k) - \sum_{i=1}^{m}u_i^k F_i(x_k) - \sum_{i=1}^{m}u^{k+1}_i F_i(x_{k})\\
    & = \mathcal{L}_p(x_k;u^k) - \mathcal{L}_p(x_{k+1};u^{k+1}) - \sum_{i=1}^{m}u_i^k F_i(x_k), 
\end{align*}

\vspace{-0.2cm}

\noindent where the second equality follows from the definition of the Lagrangian function $\mathcal{L}_p$. On the other hand, from \eqref{eq:TayAppBound} and $u_i^k\geq 0,\; i=1:m$, we have:
\begin{align*}
    - &\sum_{i=1}^{m}u_i^k F_i(x_k) \leq - \sum_{i=1}^{m}u_i^k s_{M_q^i}(x_k;x_{k-1}) + \frac{\sum_{i=1}^{m}u_i^k (M^i_q + L^i_q)}{(q+1)!}\|x_k - x_{k-1}\|^{q+1}\\
    &\qquad\quad \leq  \frac{\sum_{i=1}^{m} (u_i^k ( M^i_q + L^i_q) + \red{\eta_2})}{(q+1)!}\|x_k - x_{k-1}\|^{q+1}\\
    &\qquad \quad\leq \frac{C_u \|M_q + L_q \| + \red{m\eta_2}}{(q+1)!}\|x_k - x_{k-1}\|^{q+1},
\end{align*}
\red{where the second inequality follows from \cref{th:kkt}}. Hence, our statement follows by combining these three inequalities.
\end{proof}

\noindent 
Consider the following Lyapunov function:
\begin{align*}
\xi_p (x;u;z) :=  \mathcal{L}_p(x;u) + \frac{\theta_1}{(p+1)!} \|x - z\|^{p+1} + \frac{\theta_2}{(q+1)!}\|x - z\|^{q+1},   
\end{align*}
where $\theta_1,\theta_2$ are positive constants that will be defined later. The following lemma derives a  relation between the critical points of the functions $\xi_p$ and $\mathcal{L}_p$.

\begin{lemma}
For any $x,y \in\mathbb{E}$ and $u\in\mathbb{R}^m$, it holds that:
\begin{align*}
(x,u,z) \in\emph{crit}\; \xi_p \Rightarrow (x,u)\in \emph{crit}\; \mathcal{L}_p \;\;  \emph{and} \;\;  \xi_p(x;u;z) = \mathcal{L}_p(x;u).
\end{align*}
\end{lemma}


\begin{proof}
If $0\in \partial \xi_p (x;u;z) = \big(\partial_x \xi_p (x;u;z) , \nabla_u \xi_p (x;u;z) , \nabla_z \xi_p (x;u;z)\big)$, then:
\begin{align*}
    &0 \in \partial_x \xi_p(x;u;z) = \partial_x \mathcal{L}_p(x;u) + \left(\frac{\theta_1}{p!}\|x-z\|^{p-1} + \frac{\theta_2}{q!}\|x-z\|^{q-1} \right) (x - z),\\
    &0 = \nabla_u \xi_p(x;u;z) = \nabla_u \mathcal{L}_p(x;u),\\
    &0 = \nabla_z \xi_p(x;u;z) = \left(\frac{\theta_1}{p!}\|x - z\|^{p-1} + \frac{\theta_2}{q!}\|x - z\|^{q-1} \right)(z - x).
\end{align*}
Hence, from the last equality we get that $z = x$. This implies that $0 \in \partial \mathcal{L}_p(x;u)$ and $\xi_p(x;u;z) = \mathcal{L}_p(x;u)$, which proves our assertion. 
\end{proof}
Up to this stage, we have not considered any assumption on the constant $M$ given in subproblem \eqref{eq:rn}. Hence, by restricting the choice of this constant, we can derive the following descent in the Lyapunov function $\xi_p$.

\begin{lemma}\label{lem:4}
Let the assumptions of \cref{th:1} hold. Define $\theta_1 = \frac{M_p - L_p}{2}$, $\theta_2 = 2 C_u \|M_q + L_q\| + \red{m\eta_2}$ and $M = 3 C_u \|M_q + L_q\| + \red{2m\eta_2}$. Then, we have:
\begin{align*}
    \xi_p(x_{k+1};u^{k+1};x_k) & - \xi_p(x_k;u^k;x_{k-1}) \leq \!- \frac{(M_p \!-\! L_p)}{2(p\!+\!1)!}  \! \left(\|x_{k+1} \!-\! x_k\|^{p+1} \!+\! \|x_k \!-\! x_{k-1}\|^{p+1}\right)\\
    &\qquad\qquad  -\left(\frac{C_u \|M_q + L_q\|}{(q+1)!}\right)\left(\|x_{k+1} \!- x_k\|^{q+1} + \|x_k \!- x_{k-1}\|^{q+1}\right).
\end{align*}
\end{lemma}
\begin{proof}
We have:
\begin{align*}
    &\xi_p(x_{k+1};u^{k+1};x_k) - \xi_p(x_k;u^k;x_{k-1}) \\
    & = \mathcal{L}_p(x_{k+1};u^{k+1}) + \frac{\theta_1}{(p+1)!} \|x_{k+1} - x_k\|^{p+1} + \frac{\theta_2}{(q+1)!} \|x_{k+1} - x_k\|^{q+1} \\
    &\quad - \mathcal{L}_p(x_k;u^k) - \frac{\theta_1}{(p+1)!} \|x_k \!-\! x_{k-1}\|^{p+1} - \frac{\theta_2}{(q+1)!} \|x_k \!-\! x_{k-1}\|^{q+1} \\
     &\stackrel{\eqref{eq:th3}}{\leq}  - \frac{(M_p - L_p) - \theta_1}{(p+1)!}\|x_{k+1} - x_k\|^{p+1}  - \frac{\theta_1}{(p+1)!}\|x_k - x_{k-1}\|^{p+1}\\
    &- \left(\frac{M + \sum_{i=1}^{m}u_i^{k+1}(M_q^i - L_q^i) - \red{m\eta_2} - \theta_2}{(q+1)!} \right)\!\|x_{k+1} \!-\! x_k\|^{q+1}\\
    & - \left( \frac{\theta_2 - C_u \|M_q + L_q\| - \red{m\eta_2}  }{(q+1)!}\!\right)\!\|x_k \!-\! x_{k-1}\|^{q+1}.
\end{align*}
Then, it follows that:
\begin{align*}
     &\xi_p(x_{k+1};u^{k+1};x_k) - \xi_p(x_k;u^k;x_{k-1}) \leq - \frac{(M_p - L_p)}{2(p+1)!}\|x_{k+1} - x_k\|^{p+1} \\
     &\qquad\qquad  - \left(\frac{C_u \|M_q + L_q\| + \sum_{i=1}^{m}u_i^{k+1}(M_q^i - L_q^i)}{(q+1)!}\right)\|x_{k+1} - x_k\|^{q+1} \\
    &\qquad\qquad -\left(\frac{C_u \|M_q + L_q\|}{(q+1)!}\right)\|x_k- x_{k-1}\|^{q+1} - \frac{(M_p - L_p)}{2(p+1)!}\|x_k - x_{k-1}\|^{p+1}.
\end{align*}

\noindent Hence, our statement follows.
\end{proof}

\begin{remark}
The main difficulty in getting the descent in the Lagrangian function $\mathcal{L}_p(x;u)$ is the extra positive term that depends on the  multipliers (see Theorem \ref{th:3}). We overcome this challenge by introducing a new Lyapunov function $\xi_p$ for which we can establish the strict descent.
\end{remark}

\noindent Define the following constants $\beta_1 =\left(C_1 + \frac{M_p - L_p}{2p!}\right) $ and \\ $\beta_2 = \left( C_2 + \frac{2D(\|M_q + L_q\| + \red{\eta_3})}{(q+1)!} +  \frac{2C_u \|M_q + L_q\| + \red{m\eta_2}}{q!}\right) $. Next, we establish a  bound on the (sub)gradient of the Lyapunov function $\xi_p$:
\begin{lemma}\label{lem:5}
    Let the assumptions of Lemma \ref{lem:4} hold. Then, there exists $G_{k+1}\in \\ \partial \xi_p (x_{k+1};u^{k+1};x_k)$ such that we have the following bound:
    \begin{align*}
        \|G_{k+1}\|\leq \beta_1 \|x_{k+1} - x_k\|^p + \beta_2 \|x_{k+1} - x_k\|^q + \red{\eta_1 \|x_{k+1} - x_k\|^{\min(p,q)}}.
    \end{align*}
\end{lemma}
\begin{proof}
We have:
 \begin{align*}
  \nabla _z \xi_p(x_{k+1};u^{k+1};z)_{z=x_k}  &= \frac{M_p - L_p}{2p!}\|x_{k+1} - x_k\|^{p-1}(x_{k+1} - x_k) \\
  &\quad + \frac{2C_u \|M_q + L_q\| + \red{m\eta_2}}{q!}\|x_{k+1} - x_k\|^{q-1}(x_{k+1} - x_k).  
 \end{align*}

\noindent Then, it follows that:

 \begin{align*}
  \|\nabla _z \xi_p(x_{k+1};u^{k+1};z)_{z=x_k}\|  &\leq \frac{M_p - L_p}{2p!}\|x_{k+1} - x_k\|^p\\
  &\qquad
   +  \frac{2 C_u \|M_q + L_q\| + \red{m\eta_2}}{q!}\|x_{k+1} - x_k\|^q.       
 \end{align*}
Further, for $\Lambda_{k+1}\in\partial h(x_{k+1})$ given in Theorem \ref{th:2}, we have:
 \begin{align*} 
 &\nabla F(x_{k+1}) +\Lambda_{k+1} + \sum_{i=1}^{m} u_i^{k+1} \nabla F_i(x_{k+1})\in \partial_x\xi_p(x;u^{k+1};x_k)_{x=x_{k+1}} \\ &\left\|\nabla F(x_{k+1}) +\Lambda_{k+1} + \sum_{i=1}^{m} u_i^{k+1} \nabla F_i(x_{k+1})\right\| \leq \mathcal{M}(x_{k+1}),
  \end{align*}

\noindent where the last inequality follows from definition of $\mathcal{M}(x_{k+1})$. From \eqref{eq:fineq_c} we have:

 \begin{align*}
     F_i(x_{k+1}) &\leq s_{M_q^i}(x_{k+1};x_k) + \frac{M_q^i + L_q^i}{(q+1)!}\|x_{k+1} - x_k\|^{q+1}\\
     &\leq \frac{M_q^i + L_q^i  + \red{\eta_3}}{(q+1)!}\|x_{k+1} - x_k\|^{q+1},  
 \end{align*}

 \noindent where the last inequality follows from Assumption \ref{th:kkt}.  Finally,  we get: 
\begin{align*}
    &\|\nabla_u \xi_p (x_{k+1};u;x_k)_{u=u^{k+1}}\| = \|\left(F_1(x_{k+1}),\cdots,F_m(x_{k+1})\right)\|\\
    &\leq \frac{\|L_q + M_q\| + \red{m\eta_3}}{(q+1)!}\|x_{k+1} - x_k\|^{q+1}\\
    &\leq \frac{2D(\|L_q + M_q\| + \red{m\eta_3})}{(q+1)!}\|x_{k+1} - x_k\|^q.
\end{align*}
Denote $G_{k+1} = \left(\nabla F(x_{k+1}) +\Lambda_{k+1} + \sum_{i=1}^{m} u_i^{k+1} F_i(x_{k+1});\nabla_u \xi_p (x_{k+1};u^{k+1};x_k)\right.;\\\left. \nabla_z \xi_p(x_{k+1};u^{k+1};z)_{z=x_k}\right)$. Then,  combining the last three inequalities we get:
 \begin{align*}
     \|G_{k+1}\| &\leq \frac{M_p - L_p}{2p!} \|x_{k+1} - x_k\|^p + \frac{2C_u\|M_q + L_q\| + \red{m\eta_2}}{q!}\|x_{k+1} - x_k\|^q\\
     &\quad + \mathcal{M}(x_{k+1}) + \frac{2D(\|M_q + L_q\| + \red{m\eta_3})}{(q+1)!}\|x_{k+1} - x_k\|^q \\
     &\leq \left(C_1 + \frac{M_p - L_p}{2p!}\right)\|x_{k+1} - x_k\|^p + \red{\eta_1 \|x_{k+1} - x_k\|^{\min(p,q)}}\\ 
      &\quad + \left(C_2 + \frac{2D(\|M_q + L_q\|+\red{m\eta_3})}{(q+1)!} + \frac{2C_u\|M_q + L_q\| + \red{m\eta_2}}{q!} \right)\|x_{k+1} - x_k\|^q, 
 \end{align*}
where the last inequality follows from Theorem \ref{th:2}. This proves our assertions.
\end{proof}
From Lemma \ref{lem:4}, we have that $(\xi_p(x_k,u^k,x_{k-1}))_{k\geq 1}$ is monotonically nonincreasing. Since $\xi_p$ is continuous, then it is bounded from below and hence $(\xi_p(x_k,u^k,x_{k-1}))_{k\geq 1}$ convergences, let us say to $\xi_p^*$.
For simplicity, we assume that $p\leq q$ and denote $ S_k: =  \xi_p(x_k;u^k;x_{k-1}) - \xi_p^*$. Denote $\Gamma :=\left(\beta_1 + \red{\eta_1} + B_2D^{q-p}\right)^{\frac{p+1}{p}}\frac{2(p+1)!}{M_p - L_p}$. Next, we establish global convergence.
\begin{theorem}\label{th:kl}
Let the assumptions of Lemma \ref{lem:4} hold and let $(x_k)_{k\geq 0}$ be generated by MTA algorithm. Then, the following hold:
\begin{enumerate}
    \item If $\xi_p$ satisfies the KL property at $(x^*,u^*,x^*)$, where $u^*$ is a limit point of the bounded sequence $(u^k)_{k\geq 1}$, and $x^*$ is a limit point of $(x_k)_{k\geq 0}$, then the hull sequence $(x_{k})_{k\geq 0}$ converges to $x^*$ and there exists $k_1\geq 1$ such that:

\vspace{-0.3cm}
    
    \begin{align*}
        \|x_k - x^*\|\leq \rho \max\left(\kappa(S_k) , S_k^{\frac{1}{p+1}}\right) \;\;  \forall k\geq k_1.
    \end{align*}

    
    \item Moreover, if $\xi_p$ satisfies KL with $\kappa(s) = s^{1-\nu}$, where $\nu\in[0,1)$, then the following rates hold:
    \begin{enumerate}
        \item If $\nu = 0$, then $x_k$ converges to $x^*$ in a finite number of iterations. 
        \item If $\nu\in \left(0,\frac{p}{p+1}\right]$, then we have the following linear rate:
        
\vspace{-0.3cm}

            \begin{align*}
        \|x_k - x^*\| \leq \rho \left(\frac{\Gamma^{\frac{p}{\nu(p+1)^2}}}{(1+\Gamma^{\frac{p}{\nu(p+1)}})^{\frac{1}{p+1}}}\right)^{k-(1+k_1)}S_0^{\frac{1}{p+1}} \;\; \forall k> k_1. 
    \end{align*}

    \vspace{-0.3cm}
    
        \item If $\nu \in\left(\frac{p}{p+1},1\right)$,  then there exists $\alpha>0$ such that we have the following sublinear rate:

\vspace{-0.6cm}
        
    \begin{align*}
    \|x_k - x^*\| \leq \frac{\rho\alpha^{1-\nu}}{(k-k_1)^{\frac{p(1-\nu)}{\nu(p+1)  - p}}}\; \; \forall k > k_1.    
  \end{align*}
    \end{enumerate}
\end{enumerate}
\end{theorem}

\begin{proof}
For simplicity denote $\xi_p^k = \xi_p(x_k;u^k;x_{k-1})$ and consider $p\leq q$ (the case where $q\leq p$ is similar). From Lemma \ref{lem:4}, we have:
\begin{align}\label{eq:desS}
    \|x_{k+1} - x_k\|^{p+1} \leq \frac{2(p+1)!}{M_p - L_p} (\xi_p^k - \xi_p^{k+1})
     =  \frac{2(p+1)!}{M_p - L_p} (S_k - S_{k+1}).
\end{align}
    Further, since $\xi_p$ satisfies the inequality \eqref{eq:kl}, then there exists an integer $k_1$ and $G_k\in \partial  \xi_p(x_k;u^k;x_{k-1})$ such that for all $k\geq k_1$ we have:
    \begin{align*}
      &\|x_{k+1} - x_k\|^{p+1} \leq  \|x_{k+1} - x_k\|^{p+1} \kappa^{'}(S_k)\|G_k\|\\
      &\leq \frac{2(p+1)!}{M_p - L_p} \kappa^{'}(S_k)(S_k - S_{k+1})\|G_k\|
      \leq \frac{2(p+1)!}{M_p - L_p} (\kappa(S_k) - \kappa(S_{k+1}))\|G_k\|\\
      &\leq \frac{2(p+1)!}{M_p - L_p}\left(\kappa\left(S_k\right) - \kappa\left(S_{k+1}\right)\right) (\beta_1 + \red{\eta_1} + \beta_2 D^{q-p}) \|x_k - x_{k-1}\|^p\\
      &= \frac{2(\beta_1 +\red{\eta_1} + \beta_2 D^{q-p})(p+1)!}{M_p - L_p}(\kappa(S_k) - \kappa(S_{k+1})) \|x_k - x_{k-1}\|^p,
    \end{align*}
where the second inequality follows from \eqref{eq:desS}, the third inequality follows from $\kappa$ is concave, and the last inequality follows from Lemma \ref{lem:5}. Denote for simplicity $T = \frac{2(\beta_1 + \red{\eta_1} + \beta_2 D^{q-p})p!}{M_p - L_p}$. We further get:
\begin{align*}
    \|x_{k+1} - x_k\| &\leq \Big(T(p+1)\big(\kappa(S_k) - \kappa(S_{k+1})\big) \Big)^{\frac{1}{p+1}}\|x_k - x_{k-1}\|^\frac{p}{p+1}\\
    &\leq \frac{T(p+1)}{p+1}\big(\kappa(S_k) - \kappa(S_{k+1})\big) + \frac{p}{p+1}\|x_k - x_{k-1}\|\\
    & = T\big(\kappa(S_k) - \kappa(S_{k+1})\big) + \frac{p}{p+1}\|x_k - x_{k-1}\|,
\end{align*}
where in the second inequality we use the following classical result: if $a,b$ are positive constants and $0\leq \alpha_1 ,\alpha_2 \leq 1$, such that $\alpha_1  + \alpha_2  = 1$, then  $a^{\alpha_1}b^{\alpha_2}\leq \alpha_1 a + \alpha_2 b$. Summing up the above inequality over $k \geq  k_1$, we get:
\begin{align*}
    \sum_{k\geq k_1} \|x_{k+1} - x_k\| \leq (p+1)T \kappa(S_{k_1}) + p\|x_{k_1} - x_{k_{1}-1}\|.
\end{align*}
Hence, it follows that $(x_k)_{k\geq 0}$ is a Cauchy sequence and thus converges to $x^*$. Further:
\begin{align*}
    &\|x_k - x^*\| \leq  \sum_{t\geq k} \|x_{t+1} - x_t\|\leq   (p+1)T\kappa(S_k) + p\|x_k - x_{k-1}\|\\
    &\leq (p+1)T \kappa(S_k) + p\left(\frac{2(p+1)!}{M_p - L_p}\right)^{\frac{1}{p+1}}S_{k-1}^{\frac{1}{p+1}}
    \leq \rho\max\left(\kappa(S_k), S_{k-1}^{\frac{1}{p+1}}\right),
\end{align*}
where the third inequality follows from \eqref{eq:desS} and $S_k \geq 0$. The last inequality is straightforward by introducing $\rho = 2\max\left((p+1)T, p\left(\frac{2(p+1)!}{M_p - L_p}\right)^{\frac{1}{p+1}} \right)$. Thus, we have:
\begin{align}\label{eq:rt_kl}
     \|x_k - x^*\| \leq \rho\max\left(\kappa(S_k), S_{k-1}^{\frac{1}{p+1}}\right).    
\end{align}
Let us assume now that $\kappa(s) = s^{1-\nu}$, where $\nu \in[0,1)$. Then, it follows that:
\begin{align}\label{eq:rt_kl1}
    \|x_k - x^*\| \leq \rho \max\left(S_k^{1-\nu},S_{k-1}^{\frac{1}{p+1}}\right).
\end{align}
Further, from the KL property \eqref{eq:kl} and Lemma \ref{lem:5}, we have:
\begin{align*}
  S_k^{\nu}\leq (1-\nu) \|G_k\|\leq (\beta_1 + \red{\eta_1} + \beta_2D^{q-p})\|x_k - x_{k-1}\|^p.
\end{align*}
Hence, combining this inequality with inequality \eqref{eq:desS}, we further get:
\begin{align*}
    S_{k}^{\nu}\leq (\beta_1 + \red{\eta_1} + \beta_2D^{q-p})\left(\frac{2(p+1)!}{M_p - L_p} (S_{k-1} - S_k)\right)^{\frac{p}{p+1}}. 
\end{align*}
Recall that  $\Gamma =\left(\beta_1 + \red{\eta_1} + B_2D^{q-p}\right)^{\frac{p+1}{p}}\frac{2(p+1)!}{M_p - L_p}$, then we have the following recurrence:
\begin{align}
  S_{k}^{\frac{\nu(p+1)}{p}} \leq \Gamma(S_{k-1} - S_k).  
\end{align}
\begin{enumerate}
    \item Let $\nu = 0$. If $S_k >0$ for all $k\geq k_1$, then we have $\frac{1}{\Gamma} \leq S_{k-1} - S_k$. Letting $k\mapsto \infty$ we get $0 <\frac{1}{\Gamma}\leq 0$ which is a contradiction. Hence, there exist $k>k_1$ such that $S_k = 0$ and finally $S_k \mapsto 0$ in a finite number of steps and from \eqref{eq:rt_kl1}, $x_k\mapsto x^*$ in a finite number of iterations.
    \item Let $\nu\in(0,\frac{p}{p+1}]$, then $\frac{\nu(p+1)}{p}\leq 1$ and $1-\nu \geq \frac{1}{p+1}$. Thus using Lemma 3 in \cite{NabNec:23}, for $\theta = \frac{\nu(p+1)}{p}$, we get:
    \begin{align*}
        S_k\leq \left(\frac{\Gamma^{\frac{p}{\nu(p+1)}}}{1+\Gamma^{\frac{p}{\nu(p+1)}}}\right)^{k-k_1}S_0.
    \end{align*}
    Since we have $S_k < 1$ for all $k>k_1$ and  $S_k$ is nonincreasing, then we have $\max\left(S_k^{1-\nu}, S_{k-1}^{\frac{1}{p+1}}\right) = S_{k-1}^{\frac{1}{p+1}}$ and thus:
    \begin{align*}
        \|x_k - x^*\| \leq \rho \left(\frac{\Gamma^{\frac{p}{\nu(p+1)^2}}}{(1+\Gamma^{\frac{p}{\nu(p+1)}})^{\frac{1}{p+1}}}\right)^{k-(1+k_1)}S_0^{\frac{1}{p+1}}. 
    \end{align*}
    \item Let $\frac{p}{p+1} < \nu <1$, then $\frac{\nu(p+1)}{p}>1$ and thus using Lemma 3 in \cite{NabNec:23} for $\theta = \frac{\nu(p+1)}{p}$, there exists $\alpha >0$ such that:
    \begin{align*}
        S_k\leq \frac{\alpha}{(k-k_1)^{\frac{p}{\nu(p+1)  - p}}}.
    \end{align*}
In this case, we have $\max\left(S_k^{1-\nu}, S_{k-1}^{\frac{1}{p+1}}\right) = S_{k-1}^{1-\nu}$ and thus we have:
  \begin{align*}
    \|x_k - x^*\| \leq \frac{\rho\alpha^{1-\nu}}{(k-k_1)^{\frac{p(1-\nu)}{\nu(p+1)  - p}}}.    
  \end{align*}
\end{enumerate}
Hence, our assertions follow.
\end{proof}

\begin{remark}
In this section, we derived global convergence  rates using higher-order information for solving nonconvex problems with functional constraints provided that the Lyapunov function $\xi_p$ satisfies the KL property. Note that if $F$ and $F_i$'s, for $i=1:m$, satisfy the KL property, then $\xi_p$ also satisfies the KL property. For $p=q=1$, we recover the convergence results from \cite{YuPoLu:21,BoPa:16}. 
\end{remark}


\section{Convex convergence analysis}\label{sec:cvx}
In this section we assume that the functions $F_i$'s, for $i=0:m$, are convex functions. Then, the subproblem \eqref{eq:rn} is also convex, provided that $M_p\geq pL_p$ and $M_q^i \geq qL_q^i$ for $i=1:m$ (see Theorem 2 in \cite{Nes:20}). Hence, in this section we assume that  $x_{k+1}$ is a minimum point of the convex subproblem \eqref{eq:rn}. Note that since the Lagrangian function (see Section \ref{sub_kl}) $x\mapsto\mathcal{L}_{sp}(y;x;u)$,
is convex (provided that $M_p\geq pL_p$ and $M_q^i \geq qL_q^i$ for $i=1:m$), then from the optimality condition of $x_{k+1}$ it follows that $x_{k+1}$ is the global minimizer of the function $\mathcal{L}_{sp}(y;x_k;u^{k+1})$. Let us introduce the following constants $D_1 = \frac{M_p + L_p}{(p+1)!}$ and $D_2 = \left(\frac{M + \sum_{i=1}^{m}C_u (M_q^i + L_q^i)}{(q+1)!}\right)$, then we have the following sublinear convergence rate:

\begin{theorem}\label{th:cvx}
Let Assumptions \ref{ass:1}, \ref{ass:2} and \ref{ass:3} hold and, additionally,  $F_i$'s, for $i=0:m$, be convex functions. Let also  the sequence $(x_k)_{k\geq 0}$ be generated by MTA algorithm with $x_0 \in \cal F$, $M_p\geq pL_p, M>0$ and $M_q^i \geq qL_q^i$  for all $i=1:m$.  Then, we have the following sublinear convergence rate:
 \begin{align*}
     F(x_k) - F^* \leq
     \frac{2\max\Big((p\!+\!1)^{p+1},(q\!+\!1)^{q+1}\Big) \Big(D_1 (2D)^{p+1} + D_2 (2D)^{q+1}\Big)}{k^{\min(p,q)}} \;\;\forall k\geq 1.
 \end{align*}
\end{theorem}
\begin{proof}
If the subproblem \eqref{eq:rn} is convex and Assumption \ref{ass:3} holds, then we  proved on page 7 that the Slater's condition holds for this subproblem and consequently A-KKT conditions \eqref{eq:CAKKT} are valid with $\eta_1 = \eta_2 = \eta_3 = 0$. 
Then, we have:
\begin{align*}
    F(x_{k+1})&\stackrel{\eqref{eq:fineq},\eqref{eq:CAKKT}}{\leq} T_p^{F_0}(x_{k+1};x_k) + \frac{M_p}{(p+1)!}\|x_{k+1} - x_k\|^{p+1} + \frac{M}{(q+1)!}\|x_{k+1} - x_k\|^{q+1} \\
    &\quad+ h(x) + \sum_{i=1}^{m}u_i^{k+1}\left( T^{F_i}_q(x_{k+1} ; x_k) + \frac{M_q^i}{(q+1)!}\|x_{k+1} - x_k\|^{q+1}\right)\\
    & = \min_{x}T_p^{F_0}(x;x_k) + \frac{M_p}{(p+1)!}\|x - x_k\|^{p+1} + \frac{M}{(q+1)!}\|x - x_k\|^{q+1} \\
    &\quad + h(x) + \sum_{i=1}^{m}u_i^{k+1}\left( T^{F_i}_q(x ; x_k) + \frac{M_q^i}{(q+1)!}\|x - x_k\|^{q+1}\right)\\
    &\stackrel{\eqref{eq:fineq}}{\leq} \min_{x} F(x) + \frac{M_p + L_p}{(p+1)!}\|x - x_k\|^{p+1} + \frac{M}{(q+1)!}\|x - x_k\|^{q+1} \\
    &\quad+ \sum_{i=1}^{m}u_i^{k+1}\left( F_i(x) + \frac{(M_q^i + L_q^i)}{(q+1)!}\|x - x_k\|^{q+1}\right)\\
    &\leq \min_{\alpha \in [0,1]} \alpha F^* + (1-\alpha) F(x_k) + \alpha^{p+1}\frac{M_p + L_p}{(p+1)!}\|x_k - x^*\|^{p+1}\\
    &\quad + \frac{\alpha^{q+1} M}{(q+1)!}\|x_k - x^*\|^{q+1} + \sum_{i=1}^{m}u_i^{k+1}\Big(\alpha F_i(x^*) + (1-\alpha) F_i(x_k) \Big)  \\ 
    &\quad + \alpha^{q+1}\frac{u_i^{k+1}(M_q^i + L_q^i)}{(q+1)!}\|x_k - x^*\|^{q+1}\\
    &\leq \min_{\alpha \in [0,1]} \alpha F^* + (1-\alpha) F(x_k) + \alpha^{p+1}\frac{M_p + L_p}{(p+1)!}\|x_k - x^*\|^{p+1}\\
    &\quad + \frac{\alpha^{q+1}M}{(q+1)!}\|x_k - x^*\|^{q+1}
     + \sum_{i=1}^{m} \alpha^{q+1}\frac{u_i^{k+1}(M_q^i + L_q^i)}{(q+1)!}\|x_k - x^*\|^{q+1},
\end{align*}
where the first equality follows from $x_{k+1}$ being the global minimum of the Lagrangian function $\mathcal{L}_{sp}(y;x_k;u^{k+1})$, the last inequality follows from the fact that $x_k$ and $x^*$ are feasible for the problem \eqref{eq:problem}. Since the multipliers are bounded, we get:
\begin{align*}
   F(x_{k+1}) &\leq \min_{\alpha\in[0,1]} \alpha F^* + (1-\alpha) F(x_k) + \alpha^{p+1}\frac{M_p + L_p}{(p+1)!}\|x_k - x^*\|^{p+1}\\
   &\quad + \alpha^{q+1}\left(\frac{M + \sum_{i=1}^{m}C_u (M_q^i + L_q^i)}{(q+1)!}\right)\|x_k - x^*\|^{q+1}.
\end{align*}
Hence, we get the following relation:
\begin{align}\label{eq:for-UC}
  F(x_{k+1}) &\leq \min_{\alpha\in[0,1]} \alpha F^* \!+\! (1\!-\!\alpha) F(x_k) + \alpha^{p+1}D_1\|x_k - x^*\|^{p+1} + \alpha^{q+1}D_2\|x_k - x^*\|^{q+1},    
\end{align}
which implies that:
\begin{align}\label{eq:cv-1}
   F(x_{k+1})\leq \min_{\alpha\in[0,1]} \alpha F^* + (1\! -\! \alpha) F(x_k) + \alpha^{p+1}D_1 (2D)^{p+1}+ \alpha^{q+1}D_2 (2D)^{q+1}.
\end{align}
If $q\leq p$, then from inequality \eqref{eq:cv-1}, we get:
    \begin{align*}
        F(x_{k+1})\leq \min_{\alpha\in[0,1]} \alpha F^* + (1 \!-\! \alpha) F(x_k) + \alpha^{q+1}\left(D_1 (2D)^{p+1} + D_2 (2D)^{q+1}\right). 
    \end{align*}
Since the previous inequality holds for all $\alpha\in[0,1]$, then we consider:
\begin{align*}
 A_k:=k(k+1)\cdots (k + q),\;\;a_{k+1}: = A_{k+1} - A_k = \frac{q+1}{k+q+1}A_{k+1},\;\;\alpha_k := \frac{a_{k+1}}{A_{k+1}}.   
\end{align*}
Therefor, we obtain:
\begin{align*}
    F(x_{k+1})& \leq \frac{a_{k+1}}{A_{k+1}}F^* + \frac{A_k}{A_{k+1}} F(x_k)+ \left(\frac{a_{k+1}}{A_{k+1}}\right)^{q+1}\left(D_1 (2D)^{p+1} + D_2 (2D)^{q+1}\right)\\
    &\leq \frac{a_{k+1}}{A_{k+1}}F^* + \frac{A_k}{A_{k+1}} F(x_k)+ \left(\frac{q+1}{k+q+1}\right)^{q+1}\left(D_1 (2D)^{p+1} + D_2 (2D)^{q+1}\right).
\end{align*}
Multiplying both sides with $A_{k+1}$, we get:
\begin{align*}
    A_{k+1}(F(x_{k+1}) - F^*)&\leq A_k(F(x_k) - F^*) \\
     & \quad + A_{k+1}\left(\frac{q+1}{k+q+1}\right)^{q+1}\left(D_1 (2D)^{p+1} + D_2 (2D)^{q+1}\right)\\
    &\leq A_k(F(x_k) - F^*)+ (q+1)^{q+1}\left(D_1 (2D)^{p+1} + D_2 (2D)^{q+1}\right).
\end{align*}
Summing up this inequality, we get for all $k\geq 1$:
\begin{align*}
    F(x_k) - F^* &\leq \frac{1}{A_k} k(q+1)^{q+1}\left(D_1  (2D)^{p+1} + D_2 (2D)^{q+1}\right)\\
    &\leq \frac{(q+1)^{q+1}\left(D_1 (2D)^{p+1} + D_2 (2D)^{q+1}\right)}{k^q}.
\end{align*}
 Further, if $p\leq q$, then from inequality \eqref{eq:cv-1}, we get:
\begin{align*}
            F(x_{k+1})\leq \min_{\alpha\in[0,1]} \alpha F^* + (1 - \alpha) F(x_k) + \alpha^{p+1}\left(D_1 (2D)^{p+1} + D_2 (2D)^{q+1}\right). 
\end{align*}
Following the same procedure as before, we get:
\begin{align*}
     F(x_k) - F^* \leq \frac{(p+1)^{p+1}\left(D_1 (2D)^{p+1} + D_2 (2D)^{q+1}\right)}{k^p}.   
\end{align*}

\vspace{-0.3cm}

\noindent Hence, our assertion follows.
\end{proof}

\begin{remark}
To the best of our knowledge, this is the first convergence rate using higher-order information for a convex problem with smooth functional constraints. Note that for $p=1$ and $q=1$, we recover the convergence rate in \cite{BolChe:20}. If $m=0$, then we recover the convergence rate in \cite{Nes:20}.
\end{remark}


\subsection{Uniform convex convergence analysis}

In this section, we derive (super)linear convergence in function value for the sequence $(x_k)_{k\geq 0}$ generated by MTA algorithm, provided that  the objective function $F$ is uniformly convex of degree $\theta$ with constant $\sigma$ and the constraints $F_i$'s are only convex functions. For simplicity, let us introduce the following constants $U = \frac{(D_1(2D)^{p-q} + D_2)(q+1)!}{\sigma}$ and $\Bar{U} = \frac{(D_1 + D_2 (2D)^{q-p})(p+1)!}{\sigma}$, where $D_1$ and $D_2$ are defined in Section \ref{sec:cvx}. Then, we can establish the following convergence rate in  function values:
\begin{theorem}\label{th:unif}
 Let the assumptions of Theorem \ref{th:cvx} hold. Additionally, assume that $F$ is uniformly convex of order $\theta$ with constant $\sigma$. Then, the following hold:\\
 (i) If $\theta = \min(p+1,q+1)$, then:
 \begin{align*}
     F(x_{k+1}) - F^* &\leq \max\left( \left(1 - \frac{q}{U^{\frac{1}{q}}(q\!+\!1)^{1+\frac{1}{q}}} \right)  \!,\!   \left(1 - \frac{p}{\Bar{U}^{\frac{1}{p}}(p\!+\!1)^{1+\frac{1}{p}}} \right)  \right)
      (F(x_k) - F^*).\nonumber
 \end{align*}
 (ii) If $\theta < \min(p+1,q+1)$, then:
 \begin{align*}
    F(x_{k+1}) - F^* \leq&\max\left(\frac{(D_1 D^{p-q} + D_2)\theta^{\frac{q+1}{\theta}}}{\sigma^{\frac{q+1}{\theta}}}, \frac{(D_1 + D_2D^{q - p})\theta^{\frac{p+1}{\theta}}}{\sigma^{\frac{p+1}{\theta}}}\right)\\
    &\left(F(x_k) - F^*\right)^{\frac{\min(p+1,q+1)}{\theta}}.
\end{align*}
\end{theorem}
\begin{proof}
If $q\leq p$, then, from inequality \eqref{eq:for-UC}, we get:
 \begin{align*}
     F(x_{k+1})\leq \alpha F^* + (1-\alpha)F(x_k) + \alpha^{q+1}(D_1 (2D)^{p-q} + D_2)\|x_k - x^*\|^{q+1}.
 \end{align*}
Let $\Lambda^* \in h(x^*)$. Since $F$ is uniformly convex, we get:
\begin{align*}
    F(x_k)&\geq F^* + \langle \nabla F_0(x^*) + \Lambda^* , x_k - x^* \rangle + \frac{\sigma}{(q+1)!}\|x_k - x^*\|^{q+1}\\
          &\geq  F^* + \frac{\sigma}{(q+1)!}\|x_k - x^*\|^{q+1},
\end{align*}
where the last inequality follows from the fact that $x^*$ satisfies the optimality condition: $\langle \nabla F(x^*) + \Lambda^{*} , x - x^* \rangle \geq 0,\;\forall x\in \mathcal{F}$ and that the sequence $(x_k)_{k\geq 0}$ is feasible, i.e., $x_k\in \mathcal{F}$ for all $k\geq 0$.
Combining the last two inequalities, we get:
\begin{align*}
    F(x_{k+1})\leq \min_{\alpha \in [0,1]} \alpha F^* + (1\!-\! \alpha)F(x_k) 
     + \alpha^{q+1}\frac{(D_1 (2D)^{p-q} \!+\! D_2)(q\!+\!1)!}{\sigma}(F(x_k) \!-\! F^*).
\end{align*}
Hence, it follows that:
\begin{align*}
    F(x_{k+1}) - F^* &\leq \min_{\alpha \in [0,1]} \left(1 \!-\! \alpha \!+\! \alpha^{q+1}\frac{(D_1 (2D)^{p\!-\!q} \!+\! D_2)(q\!+\!1)!}{\sigma} \right)(F(x_k) \!-\! F^*).
\end{align*}
By minimizing the right-hand side over $\alpha$, the optimal choice is:
\begin{align*}
    0\leq \alpha = \frac{1}{(q+1)^{\frac{1}{q}} U^{\frac{1}{q}}} \leq 1.
\end{align*}
Replacing this choice in the last inequality, we get:
\begin{align*}
  F(x_{k+1}) - F^* &\leq \left( 1 - \frac{1}{(q+1)^{\frac{1}{q}}U^{\frac{1}{q}}} + \frac{U}{(q+1)^{\frac{q+1}{q}}U^{\frac{q+1}{q}}} \right)(F(x_k) - F^*)\\
                   &\leq \left( 1 - \frac{q}{U^{\frac{1}{q}}(q+1)^{1+\frac{1}{q}}} \right)(F(x_k) - F^*).
\end{align*}
 Further, if $p\leq q$, then we have:
\begin{align*}
    F(x_{k+1}) \leq \alpha_k F^* + (1-\alpha_k)F(x_k) + \alpha^{p+1}(D_1 + D_2 (2D)^{q-p})\|x_k - x^*\|^{p+1}. 
\end{align*}
Since $F$ is uniformly convex of degree $p+1$, we get:
\begin{align*}
    F(x_k)\geq F^* + \frac{\sigma}{(p+1)!}\|x_k - x^*\|^{p+1}.
\end{align*}
Combining the last two inequalities and following the same procedure as in the first case, we get the following statement:
\begin{align*}
     F(x_{k+1}) - F^*\leq \left(1 - \frac{p}{\Bar{U}^{\frac{1}{p}}(p+1)^{1+\frac{1}{p}}} \right)(F(x_k) - F^*).   
\end{align*}
Hence, our first assertion holds. Further, we have:
\begin{align}\label{eq:lc1}
    F(x_k) \geq F^* + \frac{\sigma}{\theta} \|x_k - x^*\|^\theta.
\end{align}
Taking $\alpha = 1$ in inequality \eqref{eq:for-UC} we get:
\begin{align*}
    F(x_{k+1}) - F^* \leq D_1\|x_k - x^* \|^{p+1} + D_2\|x_k - x^* \|^{q+1}. 
\end{align*}
Assume $q\leq p$. Since the sequence $(x_k)_{k\geq 1}$ is bounded, then we further get:
\begin{align*}
     F(x_{k+1}) - F^* \leq (D_1 D^{p-q} + D_2)\|x_k - x^* \|^{q+1}.    
\end{align*}
Combining this inequality with \eqref{eq:lc1}, we get:
\begin{align*}
    F(x_{k+1}) - F^* \leq \frac{(D_1 D^{p-q} + D_2)\theta^{\frac{q+1}{\theta}}}{\sigma^{\frac{q+1}{\theta}}}\left(F(x_k) - f^*\right)^{\frac{q+1}{\theta}}.
\end{align*}
If $p\leq q$, then we also get:

\vspace{-0.5cm}

\begin{align*}
    F(x_{k+1}) - F^* \leq \frac{(D_1 + D_2D^{q - p})\theta^{\frac{p+1}{\theta}}}{\sigma^{\frac{p+1}{\theta}}}\left(F(x_k) - f^*\right)^{\frac{p+1}{\theta}},     
\end{align*}
which proves the second statement of the theorem.
\end{proof}

\begin{remark}
    In \cite{DoiNes:21}  the authors propose an algorithm, different from  MTA,  for solving problem \eqref{eq:problem} based on  lower approximations of the functions  $F_i$, for $i=0:m$. Moreover, in \cite{DoiNes:21} both the objective and the constraints functions  $F_i$'s are assumed uniformly convex. Under these settings,    \cite{DoiNes:21} derives linear convergence in function values for their algorithm. Our result is less conservative since we require only the objective function to be uniformly convex and the functional constraints are assumed convex. Note also that if $p=q=1$ and $h = 0$ we recover the convergence rate in~\cite{AuShTe:10}.
\end{remark}


\section{Efficient solution of subproblem for (non)convex problems}\label{sec:4}
For $p=q=1$ it has been shown that the subproblem \eqref{eq:rn} can be solved efficiently, see \cite{AuShTe:10}. In this section, we show that the subproblem \eqref{eq:rn}  for $p=q=2$ (or $p=2$ and $q=1$) can be also solved efficiently using efficient convex optimization tools.   To this end,  for $p=q=2$ and $h=0$, in order to compute $x_{k+1}$ in subproblem \eqref{eq:rn}, one needs to solve the following problem (here $M_0= M_p + M$):
\begin{align}\label{eq:c-sbp}
&\min_{x \in \mathbb{R}^n} F_0(x_k) + \langle  \nabla F_0(x_k), x - x_k \rangle + \frac{1}{2} \left\langle \nabla^2 F_0(x_k) (x - x_k), (x - x_k) \right\rangle+ \frac{M_0}{6}\|x-x_k \|^3 \\ \nonumber
&\text{s.t.}: F_i(x_k) + \langle \nabla F_i(x_k), x - x_k \rangle + \frac{1}{2} \left\langle \nabla^2 F_i(x_k) (x-x_k), (x - x_k) \right\rangle \\ \nonumber
&\qquad\qquad\qquad\qquad  + \frac{M_i}{6} \|x - x_k\|^3 \leq 0 \;\; i=1\!:\!m.\nonumber
\end{align}

\vspace{-0.2cm}

\noindent Denote $u=(u_0,u_1,\cdots,u_m)$. Then, this problem is equivalent to:

\vspace{-0.4cm}

\begin{align*}
    \min_{x \in \mathbb{R}^n} &\max_{\substack{u\in\mathbb{R}^{m+1}_{+} \\ u_0 = 1}}\; \sum_{i=0}^{m} u_i F_i(x_k) + \left\langle \sum_{i=0}^{m}u_i \nabla F_i(x_k), x - x_k \right\rangle \\
    &\qquad\quad + \frac{1}{2} \left\langle \sum_{i=0}^{m} u_i \nabla^2 F_i(x_k) (x - x_k), (x - x_k) \right\rangle + \frac{\sum_{i=0}^{m} u_i M_i}{6} \|x - x_k\|^3.
\end{align*}

\vspace{-0.2cm}

\noindent Further, we get:

\vspace{-0.3cm}

\begin{align*}
&\min_{x \in \mathbb{R}^n} \max_{ \substack{u\in\mathbb{R}^{m+1}_{+} \\ u_0 = 1}}\; \sum_{i=0}^{m} u_i F_i(x_k) + \left\langle \sum_{i=0}^{m}u_i \nabla F_i(x_k), x \!-\! x_k \right\rangle\\
& + \frac{1}{2} \left\langle \sum_{i=0}^{m} u_i \nabla^2 F_i(x_k) (x \!-\! x_k), (x \!-\! x_k) \right\rangle
    \! +\max_{w\geq 0} \left( \frac{w}{4}\|x \!-\! x_k\|^2 - \frac{1}{12(\sum_{i=0}^{m} u_i M_i)^2} w^3 \right).
\end{align*}

\vspace{-0.3cm}

\noindent Denote for simplicity $H(u,w) = \sum_{i=0}^{m} u_i \nabla^2 F_i(x_k) + \frac{w}{2} I$, $g(u) = \sum_{i=0}^{m}u_i \nabla F_i(x_k)$, $l(u) = \sum_{i=0}^{m} u_i F_i(x_k)$ and $\Tilde{M}(u) = \sum_{i=0}^{m}u_i M_i$. 
Then, the dual formulation of this problem takes the form:

\vspace{-0.8cm}

\begin{align*}
 \min_{x \in \mathbb{R}^n} \max_{ \substack{(u,w)\in\mathbb{R}^{m+2}_{+}\\ u_0 = 1}}\; &l(u)\! +\! \left\langle g(u), x \!-\! x_k \right\rangle \!+\! \frac{1}{2} \left\langle H(u,w) (x \!-\! x_k), (x \!-\! x_k) \right\rangle \!-\! \frac{w^3}{12(\sum_{i=0}^{m} u_i M_i)^2}.    
\end{align*}

\vspace{-0.3cm}

\noindent Consider the following notations:

\vspace{-0.6cm}

\begin{align*}
&  \theta(x,u) \!= l(u) \!+\! \langle  g(u), x\!-\!x_k \rangle \!+\! \frac{1}{2}\! \left\langle\!\! \left(\sum_{i=0}^{m} u_i \nabla^2 F_i(x_k)\right)\! (x \!-\! x_k), x \!-\! x_k \!\!\right\rangle \!+\! \frac{\Tilde{M}(u)}{6}\|x \!-\! x_k\|^3, \\
    & \beta(u,w) = l(u) -\frac{1}{2} \left\langle H(u,w)^{-1} g(u) , g(u)\right\rangle - \frac{1}{12(\Tilde{M}(u))^2} w^3, \\
    & D = \left\{ (u_0,u_1,\cdots,u_m,w)\in\mathbb{R}^{m+2}_+:\; u_0 = 1\; \text{and}\; \sum_{i=0}^{m}u_i\nabla^2 F_i(x_k) + \frac{w}{2}I\succ 0  \right\}.
\end{align*}

\vspace{-0.3cm}

\noindent Then, we have the following theorem:
\begin{theorem}
If there exists an $M_i >0$, for some $i=0:m$, then we have the following relation:

\vspace{-0.6cm}

    \begin{align*}
        \theta^*:=\min_{x\in\mathbb{R}^n}\max_{u\geq 0} \theta(x,u)  =  \max_{(u,w)\in D}\beta(u,w) = \beta^*. 
    \end{align*}

    \vspace{-0.3cm}
    
\noindent Moreover, for any $(u,w)\in D$  the  direction $x(u,w) = x_k -H(u,w)^{-1}g(u)$ satisfies:
 \begin{align}\label{eq:5.2}
     0\leq \theta(x(u,w),u) - \beta(u,w)  = \frac{\Tilde{M}(u)}{12} \left(\frac{w}{\Tilde{M}(u)} + 2r_k\right)\left(r_k - \frac{w}{\Tilde{M}(u)}\right)^2,
 \end{align}

\vspace{-0.3cm}
 
 \noindent where $r_k = \|x(u,w) - x_k\|$.
\end{theorem}

\begin{proof}
First, we  show that $\theta^*\geq \beta^*$. Indeed, using a similar reasoning as in \cite{NesPol:06}, we have:  
\begin{align*}
\theta^* & =  \min_{x \in \mathbb{R}^n} \max_{ \substack{(u,w)\in\mathbb{R}^{m+2}_{+}\\ u_0 = 1}}l(u) + \left\langle g(u), x - x_k \right\rangle + \frac{1}{2} \left\langle H(u,w) (x - x_k), (x - x_k) \right\rangle \!-\! \frac{ w^3}{12\Tilde{M}(u)^2} \\
     &\geq \max_{\substack{(u,w)\in\mathbb{R}^{m+2}_{+}\\ u_0 = 1}} \min_{x \in \mathbb{R}^n}l(u) + \left\langle g(u), x - x_k \right\rangle + \frac{1}{2} \left\langle H(u,w) (x - x_k), (x - x_k) \right\rangle\!-\! \frac{ w^3}{12\Tilde{M}(u)^2} \\
     &\geq \max_{(u,w)\in D} \min_{x \in \mathbb{R}^n}l(u) + \left\langle g(u), x - x_k \right\rangle + \frac{1}{2} \left\langle H(u,w) (x - x_k), (x - x_k) \right\rangle - \frac{ w^3}{12\Tilde{M}(u)^2}\\
     & = \max_{(u,w)\in D} l(u) -\frac{1}{2} \left\langle H(u,w)^{-1} g(u) , g(u)\right\rangle - \frac{1}{12(\sum_{i=0}^{m}u_i M_i)^2} w^3
      = \beta^*.
\end{align*}
Let $(u,w) \in D$. Then, we have
$g(u)= - H(u,w)(x(u,w) - x_k)$ and thus: 
\begin{align*}
&\theta(x(u,w),u) = l(u) +   \langle  g(u), x(u,w) - x_k \rangle +  \frac{\Tilde{M}(u)}{6}r_k^3  \\
&\quad + \frac{1}{2} \left\langle \left(\sum_{i=0}^{m} u_i \nabla^2 F_i(x_k)\right) (x(u,w)- x_k), x(u,w) - x_k \right\rangle\\
& = l(u)  - \left\langle H(u,w)(x(u,w) - x), x(u,w) - x \right\rangle + \frac{\Tilde{M}(u)}{6}r_k^3  \\
&\quad + \frac{1}{2} \left\langle \left(\sum_{i=0}^{m} u_i \nabla^2 F_i(x_k)\right) (x(u,w)- x_k), x(u,w) - x_k \right\rangle\\  
& = l(u) \!-\!\frac{1}{2}\!\! \left\langle \!\! \left(\sum_{i=0}^{m} u_i \nabla^2 F_i(x_k) + \frac{w}{2}I\right)(x(u,w) - x_k), x(u,w) - x_k \right\rangle \!-\! \frac{w}{4}r_k^2 +  \frac{\Tilde{M}(u)}{6}r_k^3 \\
& = \beta(u,w) + \frac{1}{12\Tilde{M}(u)^2}w^3 - \frac{w}{4}r_k^2 +  \frac{\Tilde{M}(u)}{6}r_k^3 \\
& = \beta(u,w) + \frac{\Tilde{M}(u)}{12}\!\left(\frac{w}{\Tilde{M}(u)}\right)^3 \!\!\!- \!\frac{\Tilde{M}(u)}{4}\!\! \left(\frac{w}{\Tilde{M}(u)}\right)\!r_k^2\!+\! \frac{\Tilde{M}(u)}{6}r_k^3 \\
& =  \beta(u,w) + \frac{\Tilde{M}(u)}{12} \left(\frac{w}{\Tilde{M}(u)} + 2r_k\right)\left(r_k - \frac{w}{\Tilde{M}(u)}\right)^2,
\end{align*}
which proves \eqref{eq:5.2}. Note that we have:
\begin{align*}
    \nabla_w \beta(u,w) & = \frac{1}{4} \|x(u,w) - x_k\|^2 - \frac{1}{4\Tilde{M}(u)^2} w^2  = \frac{1}{4} \left(r_k + \frac{w}{\Tilde{M}(u)}\right)\left(r_k - \frac{w}{\Tilde{M}(u)}\right).
\end{align*}
Therefore if $\beta^*$ is attained at some $(u^*,w^*) >0$ from $D$, then we have $\nabla \beta(u^*,w^*) = 0$. This implies $\frac{w^*}{\Tilde{M}(u^*)} = r_k(u^*,w^*)$ and by \eqref{eq:5.2} we conclude that $\theta^* = \beta^*$.
\end{proof}
\begin{remark}
Note that in nondegenerate situations \textit{the global minimum} of  nonconvex cubic problem over nonconvex cubic constraints \eqref{eq:c-sbp} can be computed by:
\begin{align*}
    x_{k+1} =x_k - H(u,w)^{-1}g(u),
\end{align*}
where recall that  $H(u,w) = \sum_{i=0}^{m} u_i \nabla^2 F_i(x_k) + \frac{w}{2} I$, $g(u) = \sum_{i=0}^{m}u_i \nabla F_i(x_k)$ and $l(u) = \sum_{i=0}^{m} u_i F_i(x_k)$, with $(u,w)$ the solution of the following dual problem:
\begin{align}\label{eq:d-sbp}
&\max_{(u,w)\in D}\; l(u) -\frac{1}{2} \left\langle H(u,w)^{-1} g(u) , g(u)\right\rangle - \frac{1}{12(\sum_{i=0}^{m}u_i M_i)^2} w^3,
\end{align}
i.e., a maximization of a concave function over a convex set $D$. Hence, if $m$ is not too large, this dual problem can be solved very efficiently by interior point methods \cite{NesNem:94}. 
\end{remark}

\begin{corollary}
If there exist $M_i >0$, then the set $D$ is nonempty and convex. If the problem \eqref{eq:d-sbp} has solution, then  strong duality holds for the subproblem \eqref{eq:c-sbp}.

\end{corollary}
In conclusion,  MTA algorithm can be implementable for $p=q=2$ even for nonconvex problems, since we can effectively compute  the global minimum $x_{k+1}$ of subproblem \eqref{eq:rn} using the powerful tools from convex optimization. Note that a similar analysis can be derived for $p=2$ and $q=1$.  Next, we show the efficiency of the MTA algorithm numerically and compare it with existing methods from literature.


\section{\red{Numerical illustrations}}\label{sec:experiments}
In this section we present numerical experiments illustrating the performance of MTA algorithm and compare it with existing algorithms from the literature. We consider an optimization problem that is based on the convex function $x\mapsto\log\left(1 + \exp(a_0^{T}x)\right)$ having gradient Lipschitz with constant $\|a_0\|^2$ and hessian Lipschitz with constant $2\|a_0\|^3$, and on the nonconvex  function $x\mapsto\log\left(\frac{(c^{T}x +e)^2}{2} +1\right)$ having also  gradient Lipschitz with constant $2\|c\|^2$ and hessian Lipschitz with constant $4\|c\|^3$. These functions appear frequently in machine learning applications \cite{ZhXuGu:18}. In order to make our problem nontrivial and highly nonconvex, we add quadratic regularizers, i.e.,  we consider the problem:
\begin{align}\label{prb_sim}
    &\min_{x} F(x) =  \log\left(1+\exp(a_0^{T}x)\right) + \frac{1}{2}x^{T}Q_0x + c_0^{T}x + d_0\\
    &\text{s.t.}:\; F_i(x) = \log\left(\frac{(a_i^{T}x + b_i)^2}{2} +1\right) + \frac{1}{2}x^{T}Q_ix + c_i^{T}x + d_i\leq 0,\;i=1:m.\nonumber
\end{align}
We generate the data $a_i,b_i,Q_i,c_i,d_i$, for $i=0:m$, randomly, where $Q_i$'s are symmetric indefinite matrices  such that the problem is strictly feasible (this is ensured by the choice $d_i < \log\left(\frac{b_i^2}{2}\right)$ for $i=1:m$). Hence, the problem \eqref{prb_sim} is nonconvex, i.e., both the objective and the constraints are nonconvex functions. Our numerical simulation are performed as follows: for  given problem  data and an initial feasible point $x_0$ we compute an approximate $F^*$ and $x^*$ solution of \eqref{prb_sim} using IPOPT \cite{WacBie:06}; then, we implement our algorithm MTA(2,1) (i.e., $p=2,q=1$), MTA(2,2) (i.e., $p=q=2$), {with the regularization parameters chosen to satisfy   $M_p > L_p$ and  $M_q^i > L_q^i$ for all $i=1:m$}, and compare with SCP \cite{MesBau:21} and MBA algorithm proposed in \cite{AuShTe:10}. \red{Recall that SCP   linearizes only the nonconvex functions in the  objective and constraints at the current iterate and (possibly) adds some quadratic regularizations, keeping the other convex functions unchanged, see \cite{MesBau:21} for mode details}. Note also that MBA is similar to MTA for $p=q=1$.  The stopping criterion are: $F(x_k) - F^*\leq 10^{-3}$ and $\max_{i=1:m}(0,F_i(x_k)) \leq 10^{-3}$ and each subproblem is solved using IPOPT ({i.e., for MTA(2,1) and MTA(2,2) we use IPOPT to solve the corresponding dual subproblem \eqref{eq:d-sbp} at each iteration}). The results are given in Table \ref{tab:1} for different choices of the problem dimension ($n$) and number of constraints ($m$). In the table we report the cpu time and number of iterations for each method. From our numerical simulations, one can observe that our MTA(2,1)  performs better than MBA and SCP  methods, although the theoretical convergence rates are the same for all three methods (see \cite{AuShTe:10,YuPoLu:21}). Moreover, increasing $p$ and $q$, e.g., algorithm MTA(2,2), leads to even  much  better performance  for our algorithm, i.e.,  MTA(2,2) is superior to  MTA(2,1), MBA and SCP, which is expected from our convergence theory. Figure 1 also shows that  increasing the approximation orders $p$ and $q$ is beneficial in  our MTA algorithm, leading to better performance than first order methods (e.g. MBA and SCP) or than MTA(2,1).  





\begin{figure}[]
     \centering
     \begin{subfigure}[b]{0.48\textwidth}
         \centering
         \includegraphics[width=6.2cm, height=5.5cm]{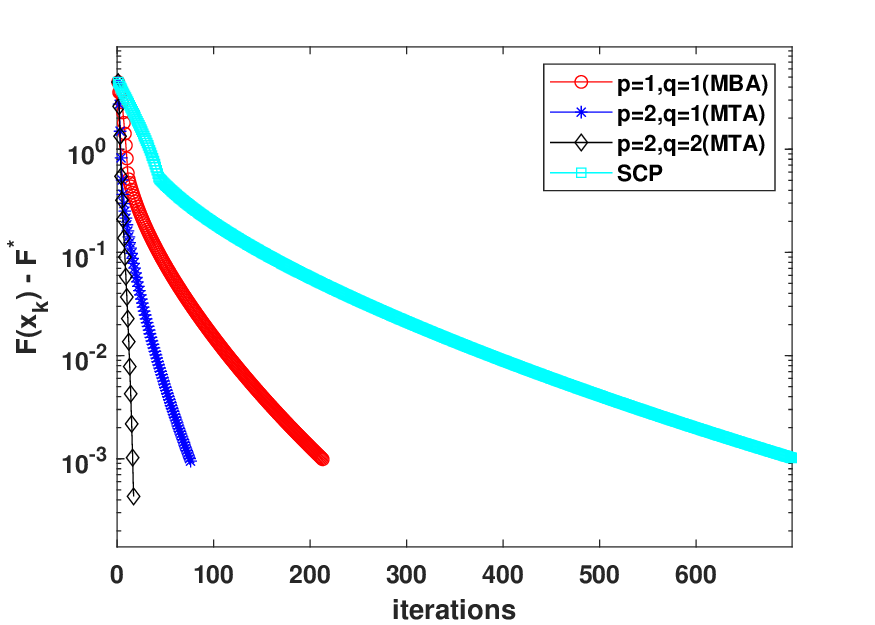}
         \label{fig:1}
     \end{subfigure}
     \begin{subfigure}[b]{0.48\textwidth}
         \centering
         \includegraphics[width=6.2cm, height=5.5cm]{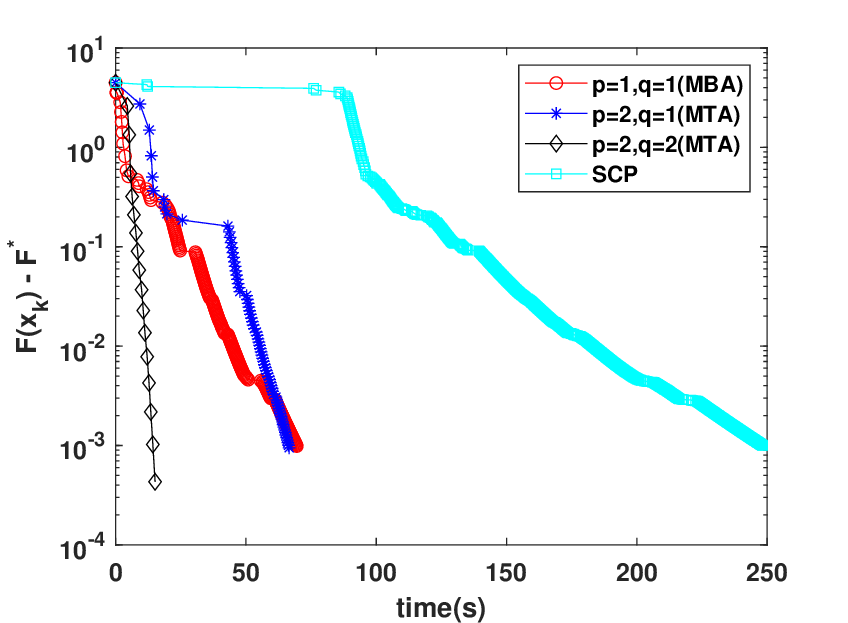}
         \label{fig:2}
     \end{subfigure}
        \caption{Behaviour of residual function for MTA with $q=2,p=1$ and with $q=p=2$, MBA and SCP along iterations (left) and time in sec (right): $n=100, m=10$.}
        \label{fig:three graphs}
\end{figure}


\begin{table}[]
\begin{tabular}{|ll|ll|ll|ll|ll|}
\hline
\multicolumn{2}{|l|}{}                          & \multicolumn{2}{l|}{SCP}      & \multicolumn{2}{l|}{MBA\cite{AuShTe:10}\;(MTA(1,1))}        & \multicolumn{2}{l|}{MTA(2,1)}   & \multicolumn{2}{l|}{MTA(2,2)}   \\ \hline
\multicolumn{1}{|l|}{n}                   & m   & \multicolumn{1}{l|}{cpu}  & iter & \multicolumn{1}{l|}{cpu} & iter & \multicolumn{1}{l|}{cpu} & iter & \multicolumn{1}{l|}{cpu} & iter \\ \hline

\multicolumn{1}{|l|}{\multirow{6}{*}{10}} & 10  & \multicolumn{1}{l|}{16.3}    &   565    & \multicolumn{1}{l|}{9.2}    &    230  & \multicolumn{1}{l|}{5.5}    &  131    & \multicolumn{1}{l|}{1.3}    &  20    \\ \cline{2-10} 

\multicolumn{1}{|l|}{}                    & 20  & \multicolumn{1}{l|}{24.5}    & 458     & \multicolumn{1}{l|}{9.8}    & 149    & \multicolumn{1}{l|}{4.2}    & 64      & \multicolumn{1}{l|}{2.5}    & 23     \\ \cline{2-10} 

\multicolumn{1}{|l|}{}                    & 50  & \multicolumn{1}{l|}{81.5}    & 401     & \multicolumn{1}{l|}{52.5}    &  184    & \multicolumn{1}{l|}{30.9}    & 118     & \multicolumn{1}{l|}{8.5}    & 18      \\ \cline{2-10} 

\multicolumn{1}{|l|}{}                    & 100 & \multicolumn{1}{l|}{428.7}    &   818    & \multicolumn{1}{l|}{145.6}    &  246    & \multicolumn{1}{l|}{55.2}    & 90     & \multicolumn{1}{l|}{25}    &  19    \\ \cline{2-10}

\multicolumn{1}{|l|}{}                    & 500 & \multicolumn{1}{l|}{1119}    &    146   & \multicolumn{1}{l|}{534.7}    &  50    & \multicolumn{1}{l|}{252.8}    & 27    & \multicolumn{1}{l|}{135.9}    &   9   \\  \cline{2-10}

\multicolumn{1}{|l|}{}                    & $10^3$ & \multicolumn{1}{l|}{1.1$\cdot 10^4$}    &  394    & \multicolumn{1}{l|}{6.7$\cdot 10^3$ }    &  188   & \multicolumn{1}{l|}{4.2$\cdot 10^3$ }    &103   & \multicolumn{1}{l|}{499.3 }    & 10    \\ \hline
\multicolumn{1}{|l|}{\multirow{5}{*}{20}}  &10  & \multicolumn{1}{l|}{149.8}    & 2477     & \multicolumn{1}{l|}{40.7}    & 700     & \multicolumn{1}{l|}{25}    &  243   & \multicolumn{1}{l|}{13.5}    &   84   \\ \cline{2-10} 

\multicolumn{1}{|l|}{}                    & 20  & \multicolumn{1}{l|}{396.5}    & 4134    & \multicolumn{1}{l|}{166.9}    &  1253     & \multicolumn{1}{l|}{69.1}    &  494    & \multicolumn{1}{l|}{25.3}    &  122    \\ \cline{2-10} 

\multicolumn{1}{|l|}{}                    & 50  & \multicolumn{1}{l|}{544.5}    &  1288    & \multicolumn{1}{l|}{266.7}    &  580    & \multicolumn{1}{l|}{151.6}    &   331   & \multicolumn{1}{l|}{37.5}    &   37   \\ \cline{2-10} 

\multicolumn{1}{|l|}{}        & 100  & \multicolumn{1}{l|}{862.4}    & 551      & \multicolumn{1}{l|}{441}    & 243     & \multicolumn{1}{l|}{264.2}    &   148   & \multicolumn{1}{l|}{72.3}    &  23    \\ \cline{2-10}

\multicolumn{1}{|l|}{}                    & 500 & \multicolumn{1}{l|}{1.9$\cdot 10^4$}    &   1078    & \multicolumn{1}{l|}{$10^4$}    &   454   & \multicolumn{1}{l|}{5052}    & 245    & \multicolumn{1}{l|}{767.2}    & 26     \\ \hline
\multicolumn{1}{|l|}{\multirow{6}{*}{100}} &  10  & \multicolumn{1}{l|}{247}    &  696   & \multicolumn{1}{l|}{69.3}    & 211     & \multicolumn{1}{l|}{66}    &  76    & \multicolumn{1}{l|}{15.5}    &  17    \\ \cline{2-10} 

\multicolumn{1}{|l|}{}                    & 20  & \multicolumn{1}{l|}{6159}    &  1179    & \multicolumn{1}{l|}{713.3}    &  306    & \multicolumn{1}{l|}{350}    &  79    & \multicolumn{1}{l|}{252}    &  28    \\ \cline{2-10} 

\multicolumn{1}{|l|}{}                    & 50  & \multicolumn{1}{l|}{2.3$\cdot 10^4$}    &   1460    & \multicolumn{1}{l|}{5974}    &  235    & \multicolumn{1}{l|}{1026}    &  25    & \multicolumn{1}{l|}{711.3}    &    18  \\ \cline{2-10} 

\multicolumn{1}{|l|}{}                    & 100  & \multicolumn{1}{l|}{5.6$\cdot 10^4$}    & 5611     & \multicolumn{1}{l|}{$1.1\cdot 10^4$}    & 1138     & \multicolumn{1}{l|}{2055}    &  89    & \multicolumn{1}{l|}{1252}    &   40   \\ \cline{2-10}

\multicolumn{1}{|l|}{}                    & 500 & \multicolumn{1}{l|}{$1.2 \cdot 10^5$}    & 1384      & \multicolumn{1}{l|}{3$\cdot 10^4$}    &      338 & \multicolumn{1}{l|}{6325}    &  67   & \multicolumn{1}{l|}{2155}    & 20     \\ \hline 
\end{tabular}
\caption{Comparison between MTA with $q=2,p=1$ and with $q=p=2$, MBA and SCP in terms of iterations and cpu time (sec) for different values $m$ and $n$.}\label{tab:1}
\end{table}

\section{Conclusions}
\label{sec:conclusions}
In this paper, we have proposed a higher-order algorithm  for solving composite problems with smooth functional constraints, called MTA. Our method uses higher-order derivatives to build a model that approximates the objective and the functional constraints. We have proven global convergence guarantees in both nonconvex and convex cases. We have also shown that our algorithm MTA is implementable and efficient in numerical simulations. \red{Our approach requires starting from a feasible point, hence it will be important to explore strategies allowing to relax this assumption.}

\medskip


\end{document}